\documentclass[a4paper,10pt,twoside]{article}

\usepackage[body={170mm,235mm}]{geometry}
\usepackage{a1}
\usepackage[english]{babel}
\usepackage[latin1]{inputenc} %permite o uso de acentos
\usepackage{amsfonts,amssymb}
\usepackage{amsmath}
\usepackage{graphicx}	

\usepackage{makeidx}
\makeindex

\def\mapright#1#2#3{\smash{\mathop{\hbox to
#3{\rightarrowfill}}\limits^{#1}_{#2}}}

\def\mapleft#1#2#3{\smash{\mathop{\hbox to
#3{\leftarrowfill}}\limits^{#1}_{#2}}}

\def\mapright#1#2{\smash{\mathop{\hbox to 0.90cm{\rightarrowfill}}\limits^{#1}_{#2}}}
\def\mapleft#1#2{\smash{\mathop{\hbox to 0.90cm{\leftarrowfill}}\limits^{#1}_{#2}}}

\def\mapleftright#1#2{\smash{\mathop{\hbox to 0.80cm{\leftarrowfill \rightarrowfill}}\limits^{#1}_{#2}}}

\title{Framed link presentations of 3-manifolds by an ${O}(n^2)$ algorithm, {III}:
  geometric complex $\mathcal{H}_n^\star$ embedded into $\mathbb{R}^3$
} 
\author{Sóstenes Lins and Ricardo Machado}

\date{\today}

\begin{document}

\maketitle

\begin{abstract}
In this final part of a 3-part paper we introduce
the pair of ``wings'' of the abstract 
PL-colored complexes $\mathcal{H}_{m}^\star$, described in the second paper.
The wings, via a weight enhanced Tutte's barycentric 
embedding of a planar map, produce the unexpected reformutation of a 3-dimensionl problem
into a 2-dimensional one. The total number of edges in 
each one of the pair of final wings is less than $8n-5$. Tutte's method is applied $O(n)$ times
to each one of the 2 wings in the final pair to assure rectilinearity of the embeddings of the planar maps,
which include the final wings. A cone construction over the final wings provides a PL-complex $\mathcal{H}_1^\diamond$,
which contain the set of 0-simplices $\{a_1, a_2,\ldots,a_f\} \cup \{b_1, b_2,\ldots,b_g\}$ (as defined in the 
second part of the article) properly fixed in $\mathbb{R}^3$.
The other 0-simplices are obtained by bisections of segments linking previously defined points.
This implies that $\mathcal{H}_n$ is PL-embedded into $\mathbb{R}^3$. 
We then conclude the surgery description of
the 3-manifold induced by the gem with its resolution by 
defining some disjoint cylinders contained in  $\mathcal{H}_{n}^\star$,
directly from the hinges (dual of the twistors of the resolution), in a 1-1 correspondence. 
The medial curves of the cylinders define the link we seek. The framing of a medial curve is
the linking number of the boundary components of the corresponding cylinder.
The analysis of the whole proccess
shows that the memory and time requirement to complete the algorithm is
$O(n^2)$. Data for the Weber-Seifert 3-manifold, which answers Jeffrey Weeks's question
is given in the appendix. It consists of a link with 142 crossings but it admits simplifications.

\end{abstract}

\section{Wings as seeds for obtaining the dual PL-complex $\mathcal{H}_n^\star$}

This is the third of 3 closely related articles.
References for the companion papers are \cite{linsmachadoA2012} and \cite{linsmachadoB2012}.

Let $\Pi_\ell$  ($\Pi_r$) be the half plane limited by the $z$-axis 
which contains $a_1=z_0z_2/2$ $(b_1=z_1z_2/2)$.
The construction of the wings and nervures of the next section are exemplified in Figs. \ref{fig:winglist01} to \ref{fig:winglist11}.

\subsection{Wings: reformulating a difficult 3D-problem 
into an easy planar one}\label{subsec:wings}

At some point in our research it became evident that what was needed to obtain the PL-complex $\mathcal{H}_n^\star$ 
was a proper embedding into $\mathbb{R}^3$ of the set of 0-simplices $\{a_1, a_2,\ldots,a_f\} \cup \{b_1, b_2,\ldots,b_g\}$.
The other 0-simplices are obtained by bisections of segments linking previously defined points. It came as a surprise to discover
that this apparently difficult 3D problem was reformulated as a plane problem for which we had at hand an easy solution,
namely Tutte's barycentric method.

We construct a 
sequence pairs of plane
graphs $\{\{\mathcal{W}^\ell_1,\mathcal{W}^r_1\},
\{\mathcal{W}^\ell_2,\mathcal{W}^r_2\},\ldots,
\{\mathcal{W}^\ell_n,\mathcal{W}^r_n\} \}.$ The $m$-th such pair
constitutes the {\em left} and {\em right wings} \index{wings} of the colored $2$-complex
$\mathcal{H}_m^\star$. The left wings are embedded into $\Pi_\ell$ and
the right wings are embedded into $\Pi_r$. 
We define $\mathcal{W}^\ell_1$ as the set of $2n$ straight line segments
$a_1z_3^1, a_1z_3^2, \ldots, a_1z_3^{2n} \subseteq \Pi_\ell$,
and $\mathcal{W}^\ell_1$ as the set of $2n$ straight line segments
$b_1z_3^1, b_1z_3^2, \ldots, b_1z_3^{2n}  \subseteq \Pi_r$.
The {\em outer triangular region of the left wings} is the plane region
spanned by $a_1, z_3^1,z_3^{2n}$. The {\em outer triangular region of 
the right wings} is the plane region spanned by
$b_1, z_3^1, z_3^{2n}$. 
The passage from $\{\mathcal{W}^\ell_{m-1},\mathcal{W}^r_{m-1}\}$
to $\{\mathcal{W}^\ell_m,\mathcal{W}^r_m\}$ in the $(m-1)$-th $bp$-move,
which we call $wbp$-move, corresponds
in $(\mathcal{H}_{m-1}, \mathcal{H}_{m})$ to either a 
0-flip that subdivides a 13-gon into two (case
where the tail of the balloon's is of color 0)
or else to a 1-flip that subdivides a 03-gon into two (case
where the tail of the balloon's is of color 1).  
At this point we need to define a tree named \index{nervure of a wing} {\em nervure of a wing}.
This is done inductively.
The first ones, $\mathcal{W}^\ell_1$ and $\mathcal{W}^r_1$ have, respectively 
the degenerated trees formed by single points $a_1$ and $b_1$ as their nervures.
In the unique wing that changes with the $bp$-move, a vertex $x_\ell$ corresponding
to either a 13-gon or else to a 03-gon (in a way to be made clear in the poof of
Lemma \ref{lem:internalpointsbigons}).
The intersection of the balloon's head and 
its tail in $\mathcal{H}^\star_m$ is a PL1-face formed by two
simplices meeting at a point $a_p$ (if the tail of the balloon is of color 1)
or $b_q$, if it is of color 0. 
Along the process we define the following auxiliary functions, with arguments
%$1 \le m \le n-1$ or
 $1 \le m \le n-1$:  
$c(m),u(m),v(m),r(m),s(m),
\ell_a(m), \ell_b(m), t_a(m), t_b(m).$
The color
of the $m$-th balloon's tail is denoted by $c(m) \in \{0,1\}$.
Let $u(m), v(m)$ be the odd and even indices of the $m$-th 
balloon's head $\nabla_{u(m)} \cup \nabla_{v(m)}$. 
Let $r(m), s(m)$ be the odd and even indices of the $m$-th 
balloon's tail given by the PL2$_{c(m)}$-face $\subseteq 
\nabla_{r(m)} \cap \nabla_{s(m)}$.
Positive integers $\ell_a(m)$ and $\ell_b(m)$ are the last $a$- and $b$-indices
in left $m$-th wing and right $m$-th wing, respectively.
Indices $p$ or $q$ in the $m$-th $bp$-move satisfy $p=t_a(m)$ or $q=t_b(m)$.
In the passage $\mathcal{H}_m^\star$ to $\mathcal{H}_{m+1}^\ast$ either 
vertex $a_p$ is replaced by $a_{p'},A_p,a_{p''}$, where $p'=\ell_a(m)+1$ and 
$p'' =\ell_a(m)+2$ or else
$b_q$ is replaced by $b_{q'},B_p,b_{q''}$, where $q'=\ell_b(m)+1$ and $q''=\ell_b(m)+2$,
depending on the color of the balloon's tail. In the first case we add two new edges
$a_{p'}A_p$ and $a_{p''}A_p$ to the nervure, in the second we add the edges
$b_{q'}B_p$ and $b_{q''}B_q$ to the nervure. In both cases, $p''=p'+1$ or
$q''=q'+1$. In the pictures the edges of the nervure
are thicker than the ones in the respective wing.
For $1\le m\le n$, and $h\in \{\ell,r\}$, the nervure of $\mathcal{W}^h_m$, 
denoted $\mathcal{N}^h_m$, is a spanning tree of the graph 
$\mathcal{W}^h_m \cup \mathcal{N}^h_m \backslash Z$, where $Z= \{z_3^j~|~ j\in \{1, \ldots, 2n\}\}.$
See Fig. \ref{fig:controlmaps01}, Fig. \ref{fig:controlmaps01novo1},  
Fig. \ref{fig:controlmaps01novo2}
and the complete sequence of figures for the $r^{24}_5$-example, 
Figs. \ref{fig:winglist01}-\ref{fig:winglist11}. 
A vertex in the tree $\mathcal{N}^h_m$ is {\em pendant} if it has degree at most 1.

\begin{lemma}
 Let $1\le m \le n$. The set of pendant vertices of $\mathcal{N}_m^\ell$ is
in 1-1 correspondence with the set of 13-gons of $\mathcal{H}_m$. 
 The set of pendant vertices in $\mathcal{N}_m^r$ is
in 1-1 correspondence with the 03-gons of $\mathcal{H}_m$.
\label{lem:internalpointsbigons}
\end{lemma}
\begin{proof}
The intersection of the $(m-1)$-th balloon's head and tail is a 
PL1-face with two 1-simplices. Their intersection is a point in $\Pi_h$.
The PL1-face dualy corresponds to a 13-gon (resp. 03-gon) in $\mathcal{H}_{m-1}$
if $h=\ell$ ($h=r$). After the 0-flip (resp. 1-flip) that produces  $\mathcal{H}_m$
the PL1-face is splitted into two, in a conformal way with the passage 
from  $\mathcal{W}^h_{m-1} \cup \mathcal{N}^h_{m-1}$ to  
$\mathcal{W}^h_m \cup \mathcal{N}^h_m$. Given this interpretation
the Lemma is easily established by induction.
\end{proof}

%--------------------
\begin{figure}[!htb]
\begin{center}
\includegraphics[width=15.4cm]{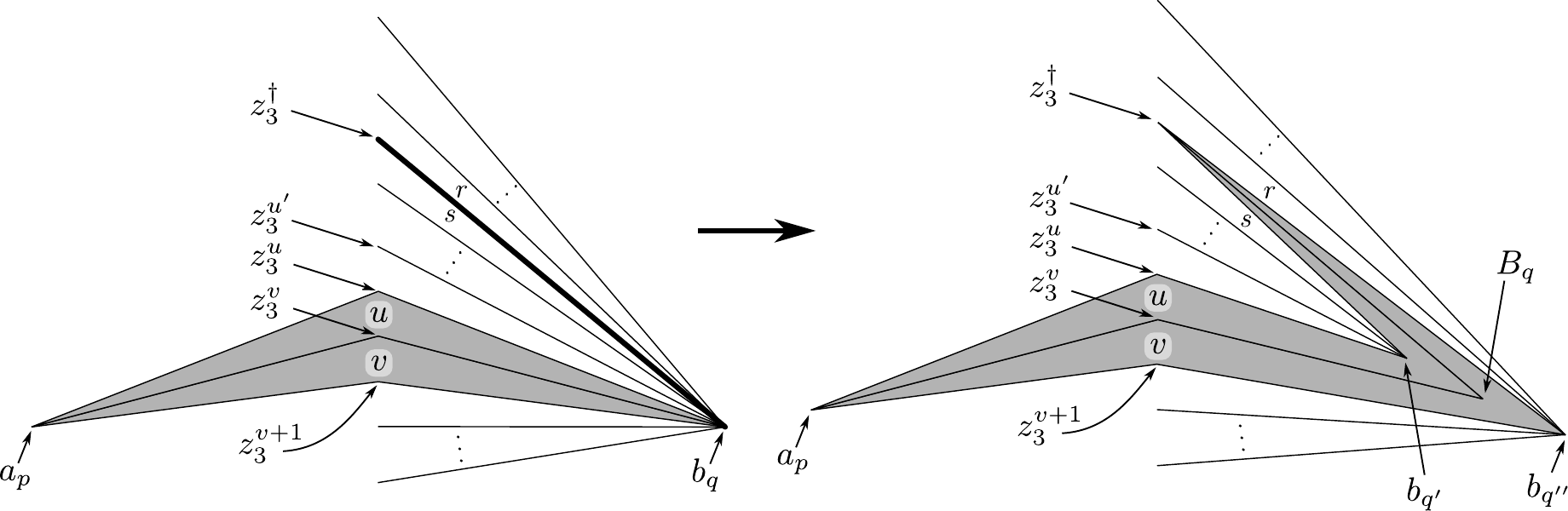}
\caption{$wbp$-move: balloon's head section is painted in gray, 
and the part of balloon's tail that intersecting the appropriate
semi-plane is depicted as a {\em thick edge}. \index{thick edge}}
\label{fig:controlmaps01}
\end{center}
\end{figure}
%--------------------

\begin{figure}[!htb]
\begin{center}
\includegraphics[width=15cm]{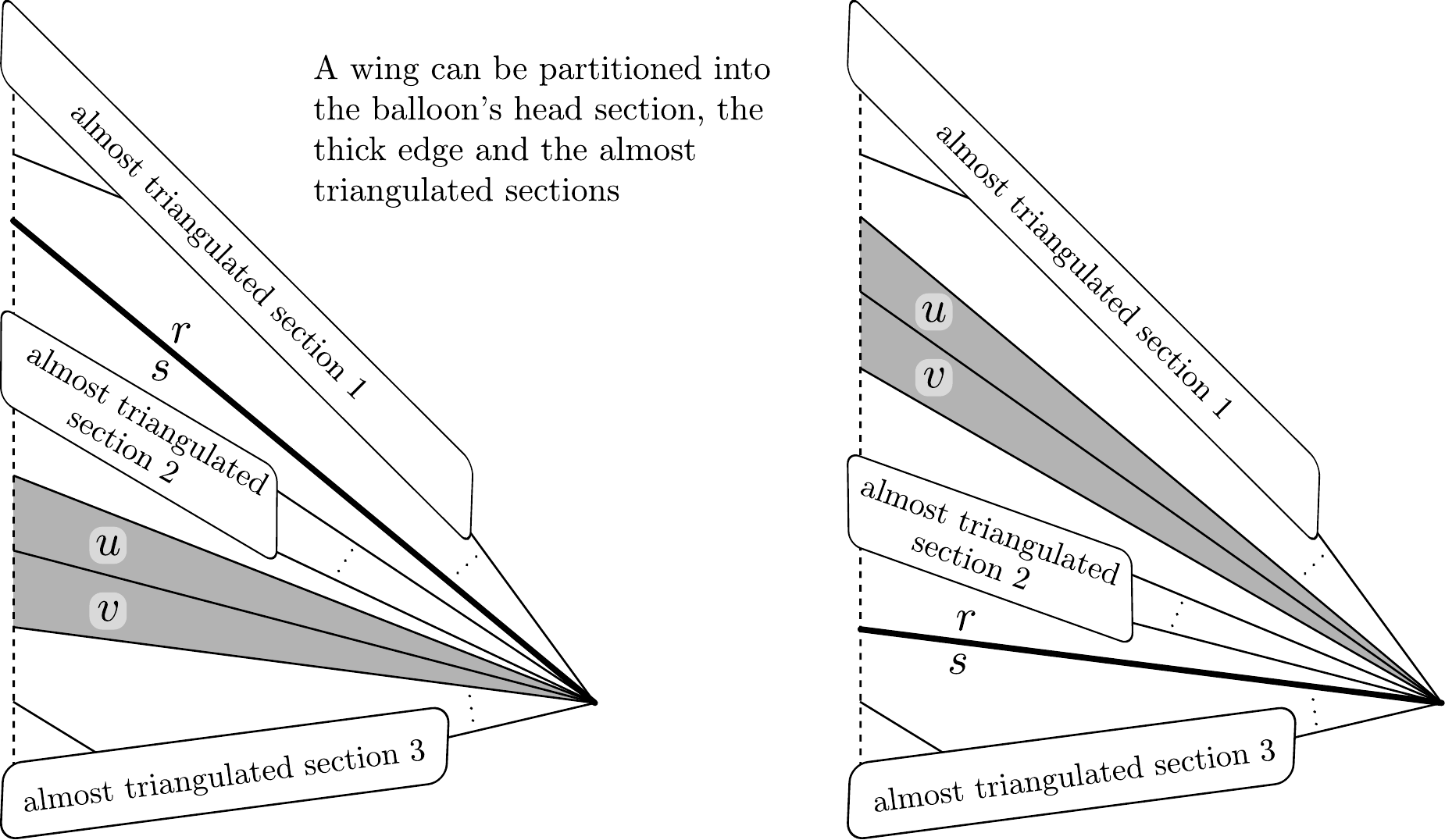}
\caption{Given a half-wing and a balloon it can be partitioned into the 
balloon's head section, thick edge and
some {\em almost triangulated} sections.}
 %meaning that some faces are bounded by plane quadrilaterals instead of triangles.}
\label{fig:controlmaps01novo1}
\end{center}
\end{figure}

\begin{figure}[!htb] 
\begin{center}
\includegraphics[width=15 cm]{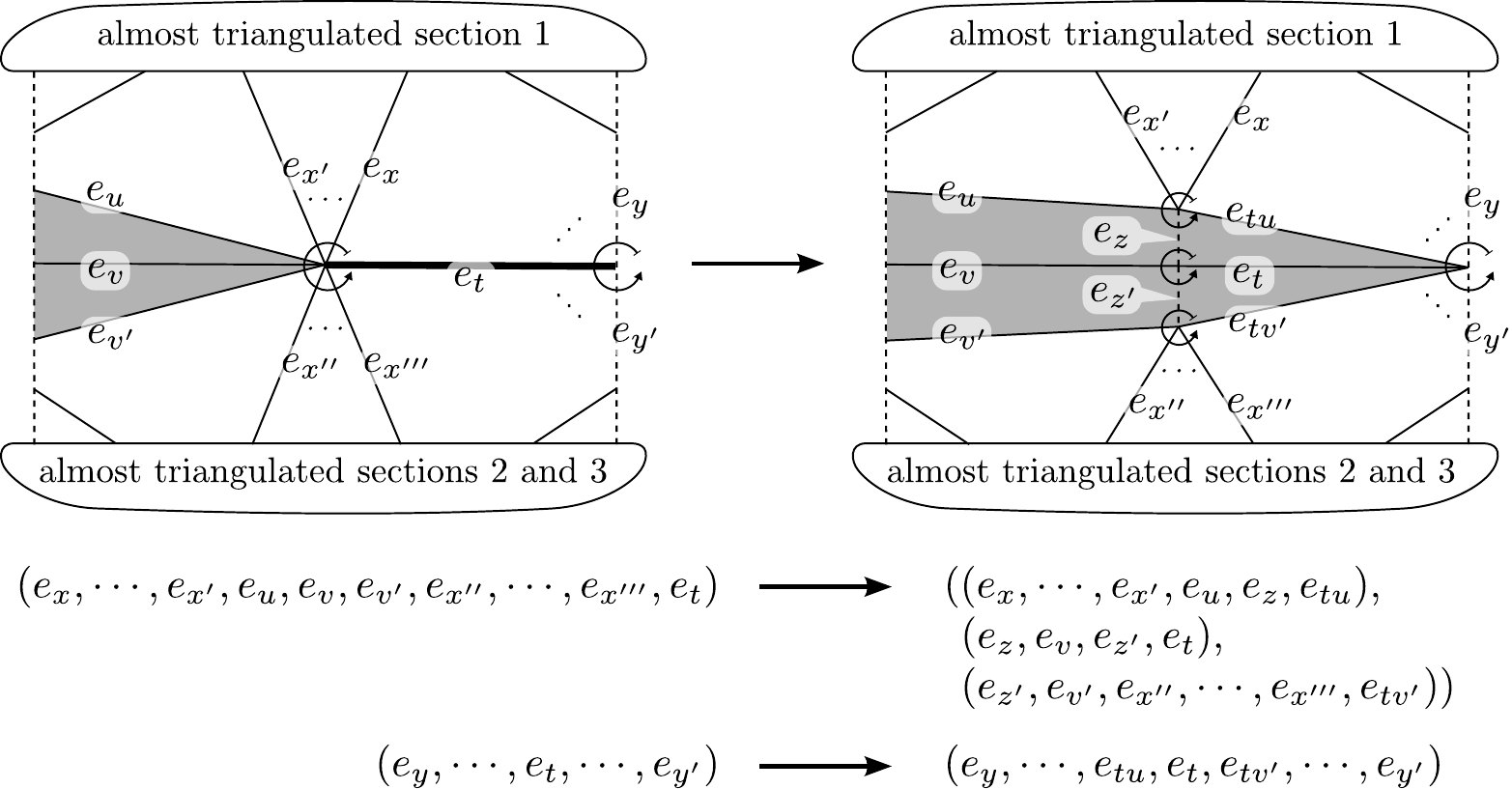}
\caption{The {\em star} \index{star} of a vertex of a graph embedded in a surface is the counterclockwise cyclic sequence of 
edges incident to the vertex (such an ordering is induced by the surface).
The set of stars is called a { \em rotation} \index{rotation} and has the 
characterizing property that each edge appears twice.
The general case of changing rotation when going from 
$\mathcal{W}_{\ell-1} \cup \mathcal{N}_{\ell-1}$ to
$\mathcal{W}_{\ell} \cup \mathcal{N}_{\ell}$ is depicted above. 
The rotation completely specifies the topological embedding.}
\label{fig:controlmaps01novo2}
\end{center}
\end{figure}

A graph is {\em rectilinearly embedded into $\mathbb{R}^3$} if 
the images of their edges are straight
line segments.
It is a straightforward application of Tutte's barycentric method (\cite{tutte1963dg},
\cite{colin2003tutte}) to obtain a rectilinear embedding of
$\mathcal{W}_n^h \cup \mathcal{N}_n^h$ which fixes the vertices in the boundary of the 
outer triangular region of $\Pi_h$, $h \in \{\ell,r\}$.
Tutte's method has an intrinsic connection with the Laplacian of graphs, see  \cite{Klarreich0412}.
We rotate $\Pi \in \{ \Pi_\ell, \Pi_r \}$
so that it becomes the $xz$-plane. After having the planar coordinates
$\Pi_h$ is rotated back to its initial position and we have the 
$\mathcal{W}_n^h \cup \mathcal{N}_n^h$ rectilinearly embedded into $\mathbb{R}^3$.
Tutte's method becomes very efficient because of Lemma \ref{lem:numberofedges}.

Tutte's method suffers of the clustering problem where vertices accumulate
in some small regions. Even though theoretically this is not 
needed, we present a heuristic of attaching weights to the edges
to improve the result. An edge with weight $k\in\mathbb{N}$ behaves
as $k$ parallel edges.
We define the weights for the edges of the wings as 1.
If the edge is in the nervure,
to calculate the weight we use the whole wing. 
Start defining these weights as 0, so from leaves to root, 
define the weight of an edge
as the weight of the vertex incident to it and closer to 
the leaves minus 1, Fig. \ref{fig:arvorecompesos}.
In Fig. \ref{fig:controlmapsnovo6} we compare the two results,
without and with weights given by three times the weights in the nervure 
of Fig. \ref{fig:arvorecompesos}, 
for obtaining the final left wing of the $r^{24}_5$-example.

Tutte's method is applied twice: to plane graphs $\mathcal{W}_n^\ell \cup \mathcal{N}_n^\ell$ and to $\mathcal{W}_n^r \cup \mathcal{N}_n^r.$
In each application we use $O(n)$ iterations to solve a linear system in $\mathbb{C}$. This is theoretically sufficient
to achive rectilineatiry (which nevertheless can be verified). As each one of the plane graphs has less than $6n-4$ edges by Lemma \ref{lem:numberofedges}, 
the total time to obtain $\mathcal{W}_n^\ell \cup \mathcal{W}_n^r$ embedded into $\Pi_\ell \cup \Pi_r$ is $O(n^2)$.

\begin{lemma}\label{lem:numberofedges}
 The number of edges of $\mathcal{W}_n^h \cup \mathcal{N}_n^h$, $h \in \{\ell,r\}$
is at most $6n-4$.
\end{lemma}
\begin{proof}
The number of 1-simplices in the left wing and in the right wing of the 
initial complex in the sequence are both $2n$. At each one of the $n-1$ $bp$-moves
we add 4 edges either to the left or to the right wing with its nervure.
Thus each one of the final left and right wings with nervures has at most $6n-4$ edges.
\end{proof}

\begin{figure}[!htb]
\begin{center}
\includegraphics[width=13cm]{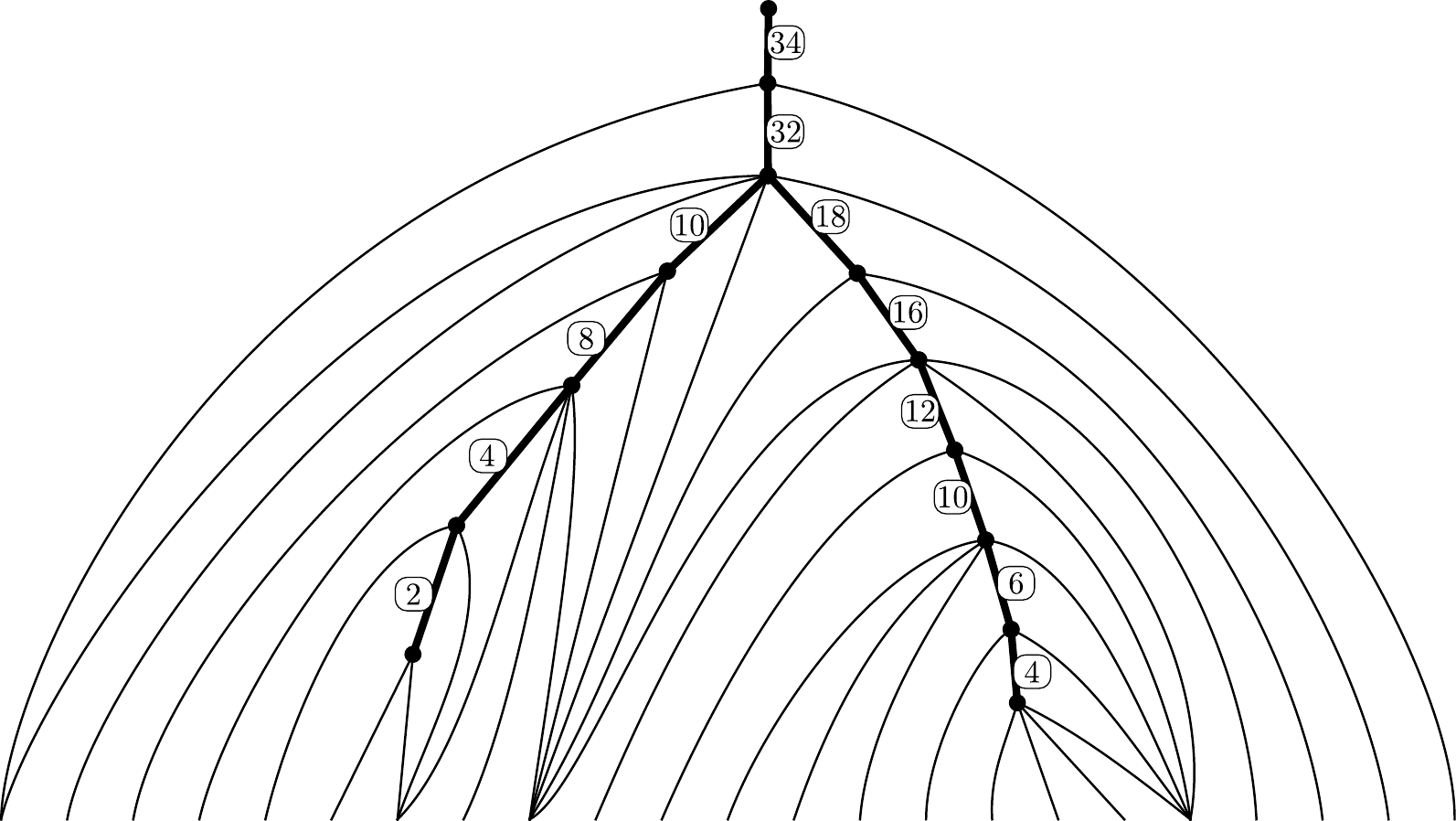}
\caption{Computing the weights for Tutte's barycentric method via the wing nervure.}
\label{fig:arvorecompesos}
\end{center}
\end{figure}

\begin{figure}[!htb]
\begin{center}
\includegraphics[width=13cm]{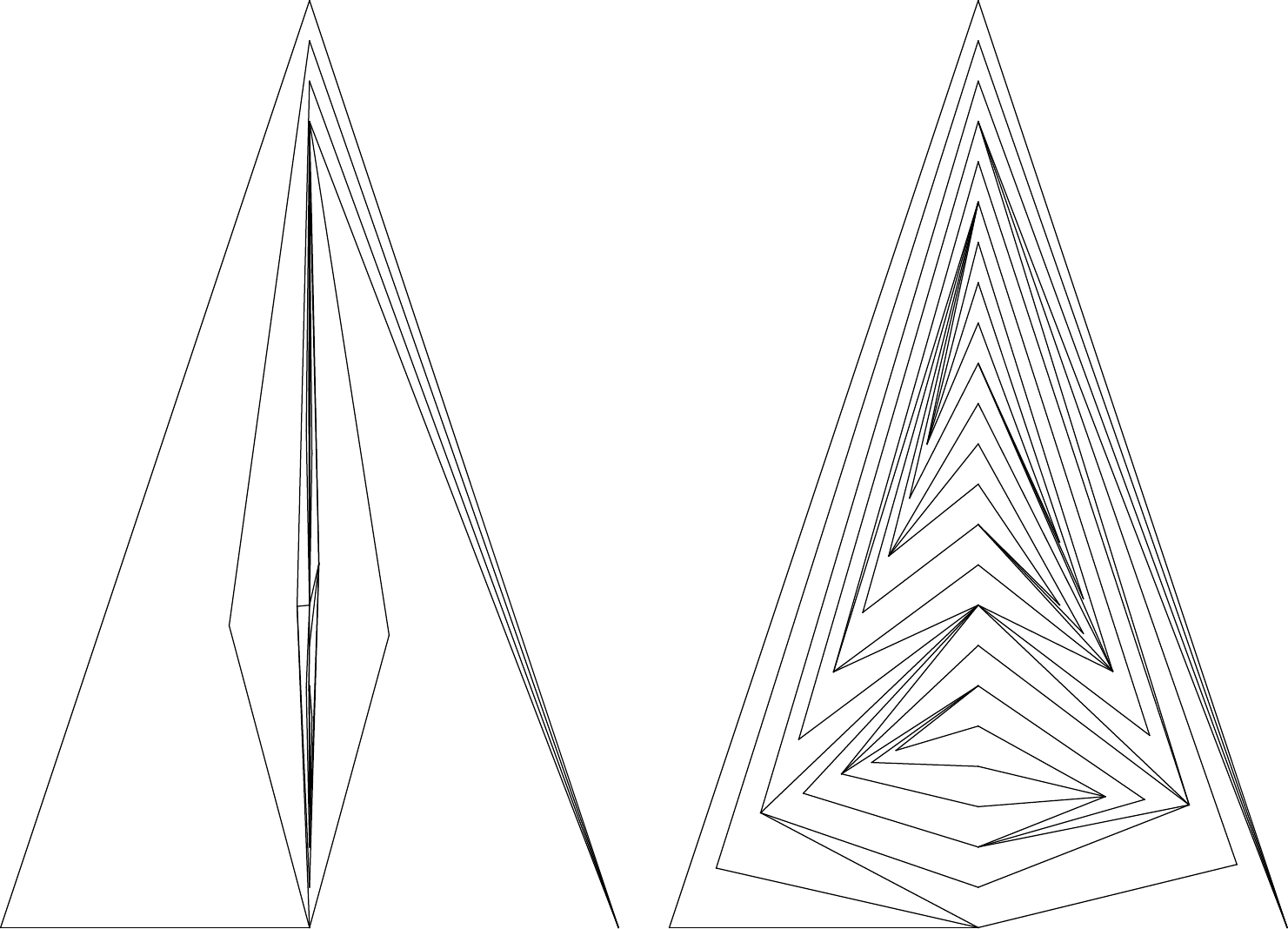}
\caption{Tutte's embbeding without and with the weights (final pair of 
wings without nervure for the $r_5^{24}$-example).}
\label{fig:controlmapsnovo6}
\end{center}
\end{figure}
%\begin{lemma} 
%\label{theo:containement}
%$$\mathcal{U}^\star_n \subseteq (z_0 \ast W_n^\ell) 
%\cup (z_2 \ast W_n^\ell) \cup (z_1 \ast W_n^r) \cup (z_2 \ast W_n^r) 
%\bigcup_{i=1}^{2n} (z_3^i \ast z_0z_1) .$$
%\end{lemma}
%\begin{proof}
%Every 2-simplex of $\mathcal{U}^\star_m$ is contained in one of the 4+2n cones.
%\end{proof}

\newpage

\subsection{Defining the PL-embedding $\mathcal{H}_{1}^\diamond$}\label{sec:rect}

%Let $\mathcal{H}_1^\diamond=\mathcal{H}_1^\star$ be.
Let $\mathcal{L}_{i+1}^\star$ be a subset of the pillow $\mathcal{P}_{i+1}^\star$, formed by the part
that comes from the tail of the balloon after the i-th $bp$-move is applied, see Fig. \ref{fig:U3}.
%, see Fig. \ref{fig:001partenova}.
%In other words, $\mathcal{U}_n^\star=\{ v_0v_3^jv_1, v_0v_4^jv_3^j, v_4^jv_2v_3^j, v_2v_5^jv_3^j, v_5^jv_1v_3^j~|~ j\in \{ 1, \ldots, 2n \} \}$ 
%and $\mathcal{L}_n^\star=\mathcal{H}_n^\star \backslash \mathcal{U}_n^\star$, see Fig. \ref{fig:001partenova}.
%Note that $\mathcal{U}^\star_n$ denote the subcomplex  
%and
%$\mathcal{L}^\star_n$ denote the subcomplex formed by the tails 
%of the balloons or the part of the pillows that came from the tails of the balloons.

\begin{figure}[!htb]
\begin{center}
\includegraphics[width=13cm]{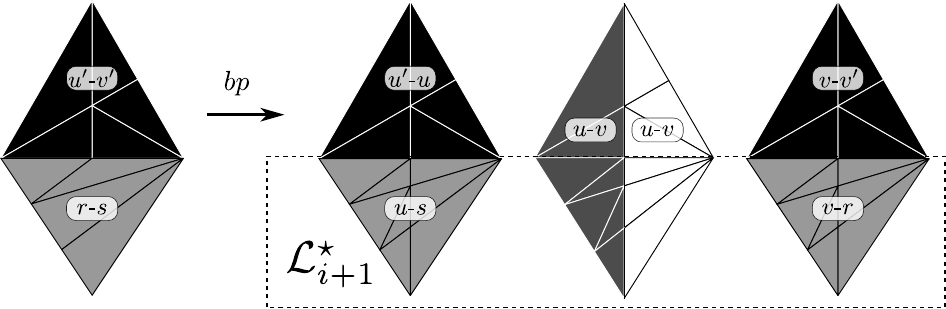}
\caption{The set $\mathcal{L}_{i+1}^\star$.} 
\label{fig:U3}
\end{center}
\end{figure}

%\begin{figure}[!htb]
%\begin{center}
%\includegraphics[width=14cm]{A.figs/001partenova.pdf}
%\caption{Two PL2$_3$-faces, a subset of a PL2$_0$- or PL2$_1$-face and a subset of a PL2$_2$-face (in the middle) at
% the pillow corresponding to the part that comes from the head of a balloon after all $bp$-moves be applied.}
%($\nabla_u$ and $\nabla_v$ sharing two 2-simplices in the pillow).} 
%\label{fig:001partenova}
%\end{center}
%\end{figure}

Let $\{x\} \cup Y \subseteq \mathbb{R}^N$, for $1\le N \in \mathbb{N}$. 
The \index{cone} {\em cone} (\cite{rourke1982introduction}) 
with vertex $x$ and base $Y$, denoted $x \ast Y \subseteq \mathbb{R}^N$,
is the union of $Y$ with all line 
segments which link $x$ to $y \in Y$.

%Let $\mathcal{U}^\star_n$ denote the subcomplex formed by the heads of 
%the balloons or the part 
%of the pillows that came from the heads of the balloons.
%Let $\mathcal{L}^\star_n$ denote the subcomplex formed by the tails 
%of the balloons or the part of the pillows that came from the tails of the balloons.
%Note that $\mathcal{H}^\star_n =  \mathcal{U}^\star_n\cup\mathcal{L}^\star_n$.

%It is now possible to use the (rectilinearly embedded into $\Pi_\ell \cup \Pi_r$) 
Now we define a PL-complex $\mathcal{H}^\diamond_1$ explicitly embedded into $\mathbb{R}^3$. 
We use the (rectilinearly embedded into $\Pi_\ell \cup \Pi_r$) 
final wings and the cone contruction to get the 
$\mathcal{H}^\diamond_1$. To this end, select a distinguished representative 
of the edges of $\mathcal{W}_n^\ell$ (resp. $\mathcal{W}_n^r$) incident to $z_3^j$ in the following way: if there is just one edge, choose it. 
Otherwise the representative is the edge whose other end has the smallest indexed upper case label.
Let $R$ denote the set of representatives.

For each $e \in R$ 
%let  $v \in \{a_k,A_{k'},b_\ell,B_{\ell'}\}$ 
%be the other end of $e$.
add the two 2-simplices $z_0\ast e$ and $z_2\ast e$ (resp. the two 2-simplices 
$z_1\ast e$ and $z_2\ast e$) to  $\mathcal{H}^\diamond_1$.  
To complete 
$\mathcal{H}^\diamond_1$ add the 
2-simplices $\{z_3^jz_1z_0 \ | \ j=1,\ldots,2n\}$.
In Fig. \ref{fig:U} the solid lines
  (the edges of $R$) and the dashed edges are part of $\mathcal{L}^\star_i$,
and are treated in next section.

\begin{figure}[!htb]
\begin{center}
\includegraphics[width=15cm]{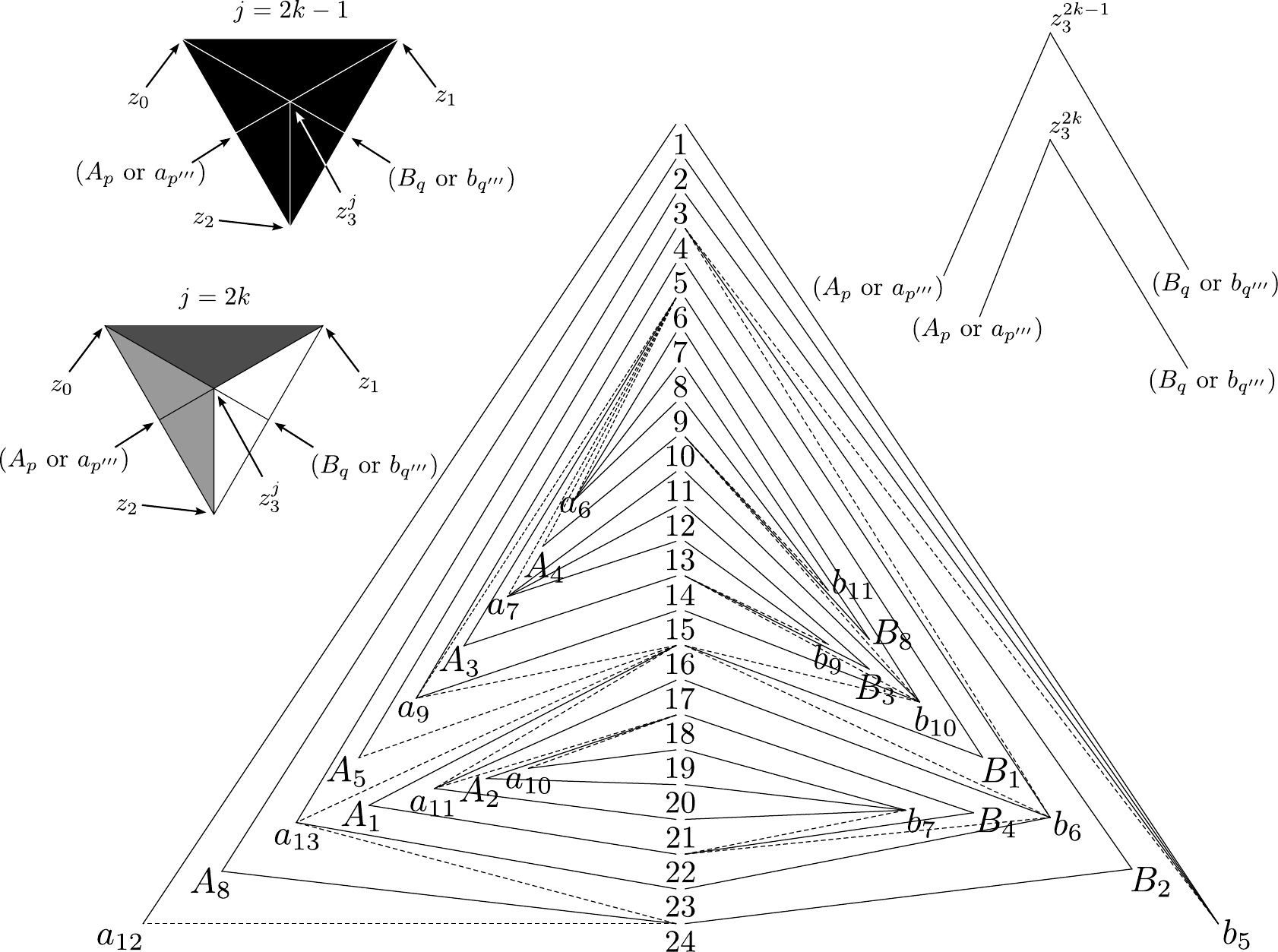}
\caption{We use the cone construction with the solid lines to obtain $\mathcal{H}_1^\diamond.$ 
Then we use the dashed lines
to obtain the information of $z_3^\dagger$ and $a_{p'}, A_p, a_{p''}$ or $b_{q'}, B_q, b_{q''}$ 
latter when obtaining $\mathcal{L}_i^\star$. Also, we have $p'''\in\{p' ,p''\}$ and  $q'''\in\{q' ,q''\}$.} 
\label{fig:U}
\end{center}
\end{figure}

\begin{proposition}
\label{cor:embedcone}
If $\mathcal{W}_m^h$ is embedded rectilinearly in $\Pi_h$, $h \in \{\ell,r\}$,
then the pair of embeddings can be 
extended to an embedding of $\mathcal{H}^\diamond_1$ into $\mathbb{R}^3$,
via the cone construction.
\end{proposition}
\begin{proof}
 Straighforward from the simple geometry of the situation.
\end{proof}
%--------------------
%\begin{figure}
%\begin{center}
%\includegraphics[scale=1.2]{A.figs/grau.pdf}
%\caption{}
%\label{fig:grau}
%\end{center}
%\end{figure}
%--------------------

\subsection{Blowing up the tails and
constructing\\ $\mathcal{H}_{2}^\diamond, \mathcal{H}_{3}^\diamond, 
\ldots, \mathcal{H}_{n}^\diamond=\mathcal{H}_{n}^\star$}

The process of replacing the embedded tail of a balloon by the corresponding
trio of PL2-faces in the pillow is denominated \index{blow up a tail} {\em the blowing up of the 
balloon's tail}.

%{\huge \bf Aqui insira seus $\alpha, \beta, \gamma$, bumps, etc}
%\begin{theorem}\label{theo:teoremadeumalinha}
% There is an $O(n)$-algorithm for blowing up a single balloon's tail.
%Thus finding $E\mathcal{L}_n^\star$ and  $E\mathcal{H}_n^\star= 
%E\mathcal{U}_n^\star \cup E\mathcal{L}_n^\star$
%take $O(n^2)$ steps.
%\end{theorem}
\begin{theorem}\label{theo:teoremadeumalinha}
 There is an $O(n)$-algorithm for blowing up a single balloon's tail.
Thus finding $\mathcal{H}_n^\star$ and take $O(n^2)$ steps.
\end{theorem}
\begin{proof}
%Here we get an explicit embedding $E\mathcal{L}_{n}^\star$.
%This finishes the description of $E\mathcal{H}_n^\star.$
%This 2-complex becomes geometrically
%consistent with the combinatorial one in the sense 
%that there are no spurious crossings. 
$\mathcal{H}_{i+1}^\diamond$ is the union of $\mathcal{H}_{i}^\diamond$ with $\mathcal{L}_{i+1}^\star$
and an $\epsilon$-change in some PL3-faces, if the rank of the type of balloon's tail of the i-th $bp$-move has rank greater than 1
(we call $\epsilon$-change because this change is small, as described below).
At the same time we update the colors of the middle layer to match the colors of the $i$-th pillow in the sequence of $bp$-moves.

%(We proceed the embed of $\mathcal{L}_i^\star$ in the same order as the $bp$-move.
%In each step from $E\mathcal{U}_i^\star$ to $E\mathcal{U}_{i+1}^\star$,
% we change the color of two 2-simplices in the medial layer, of the open balloon's head of this step in $E\mathcal{U}_i^\star$,
%to match with the medial layer of the $bp$-move and embed $\mathcal{L}_i^\star$.
%Basically $E\mathcal{U}_{i+1}^\star$ is the union of $E\mathcal{U}_{i}^\star$ with 
%$E\mathcal{L}_{i}^\star$ and a change of the color in two 2-simplices of $\mathcal{U}_{i}^\star$
%(in some steps we make changes in some tetrahedra which is already embedded to create space to the embed of $\mathcal{L}_i^\star$).
% The final embedding $E\mathcal{U}_n^\star$ is equals to $E\mathcal{H}_n^\star$,
%but $E\mathcal{U}_i^\star \neq E\mathcal{H}_i^\star$ if $i<n.$

Now we describe how to embed each kind of $\mathcal{L}_{i}^\star$ 
(explaining how to $\epsilon$-change some PL3-faces,
to get space for $E\mathcal{L}_i^\star$).

If the balloon's tail is of type $P_1$ (the case $B_1$ is analogous).
Make two copies of $P_1$, resulting in three $P_1$, but change the color of the one which will be in
the middle, and define the 0-simplices like in Fig. \ref{fig:3d2}.

\begin{figure}[!htb] 
\begin{center}
\includegraphics[width=14cm]{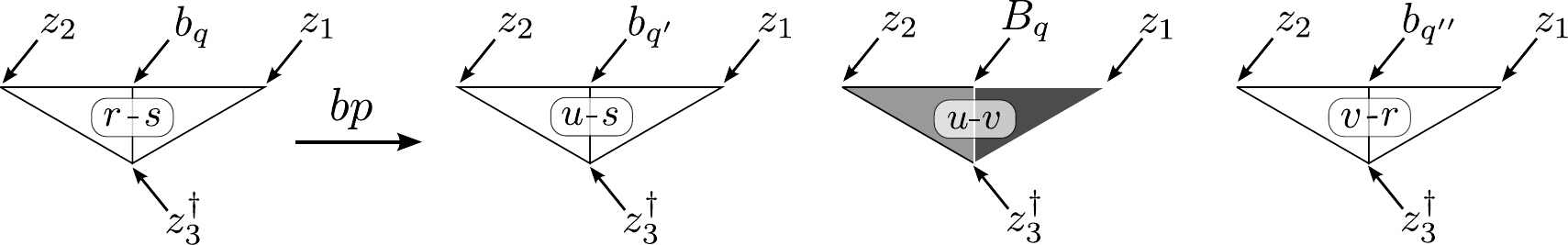}
\caption{Embedding the part of the pillow corresponding to the tail of the balloon: case $P_1$ of the tail.}
\label{fig:3d2}
\end{center}
\end{figure}
If the balloon's tail is of type $B_i$, $i>1$ (the case $P_i$ is analogous).
Make two copies of $B_i$, refine the copies and the original, resulting in three $B_i'$, 
but change the color of the one which will be in
the middle, and define the 0-simplices like in Fig. \ref{fig:3d3}.

The images $\chi_j$ we already know from previous $bp$-move, now we need to define all the images
$\alpha_j, \beta_j$ and $\gamma_j.$
Let $\beta_j$ be $\frac{z_2+\chi_{j+1}}{2}$ 
for each $j=1,\ldots, i$.
As the images $\alpha_j$ and $\gamma_j$ can be defined in analogus way, we just explain how to define each $\alpha_j$.
We know that each $\alpha_j$ is in the PL3-face $\nabla_r$. To 
define each $\alpha_j$ we need 
to reduce the PL3-face $\nabla_r$
the get enough space for the PL2-faces of color 0 and 2 of the PL3-faces $\nabla_u$ and $\nabla_v$.
Consider the PL3-face $\nabla_r$, each $\beta_j$ is already defined, so 
define each $\zeta_j$ as $\frac{z_2+\omega_{j+1}}{2}$, where $\omega_k$ is previously defined, %that they are as if we where going to refine
 see Fig. \ref{fig:nextdual3}. Define $\alpha_j$ as $\frac{\zeta_j+\beta_j}{2}$. 

\begin{figure}[!htb] 
\begin{center}
\includegraphics[width=14cm]{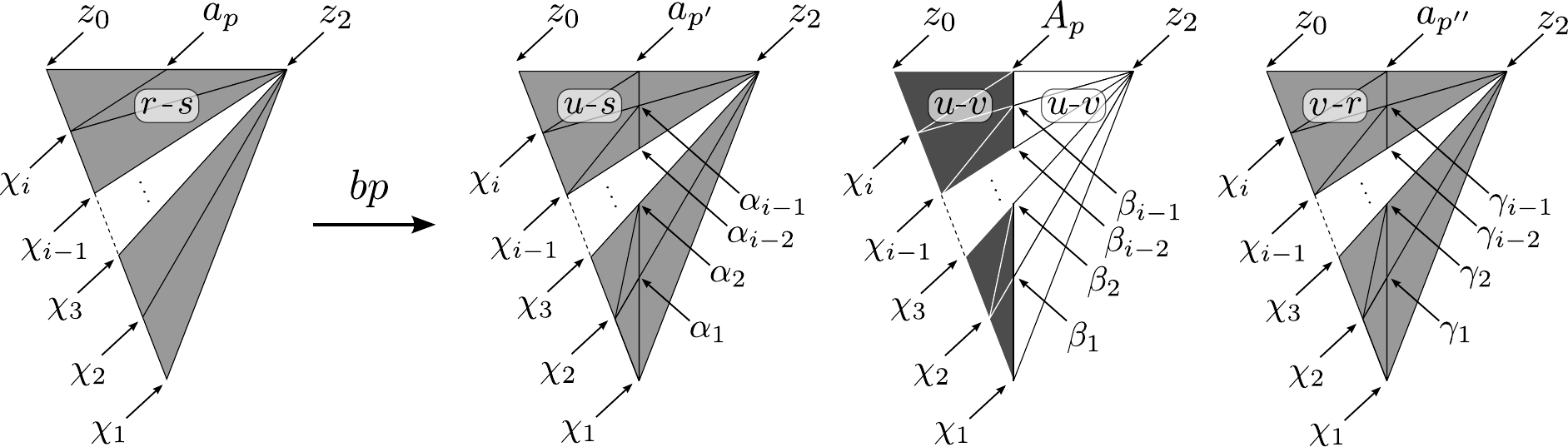}
\caption{Embedding the part of the pillow corresponding to the tail of the balloon: case $B_i$ of the tail.}
\label{fig:3d3}
\end{center}
\end{figure}

\begin{figure}[!htb] 
\begin{center}
\includegraphics[scale=0.8]{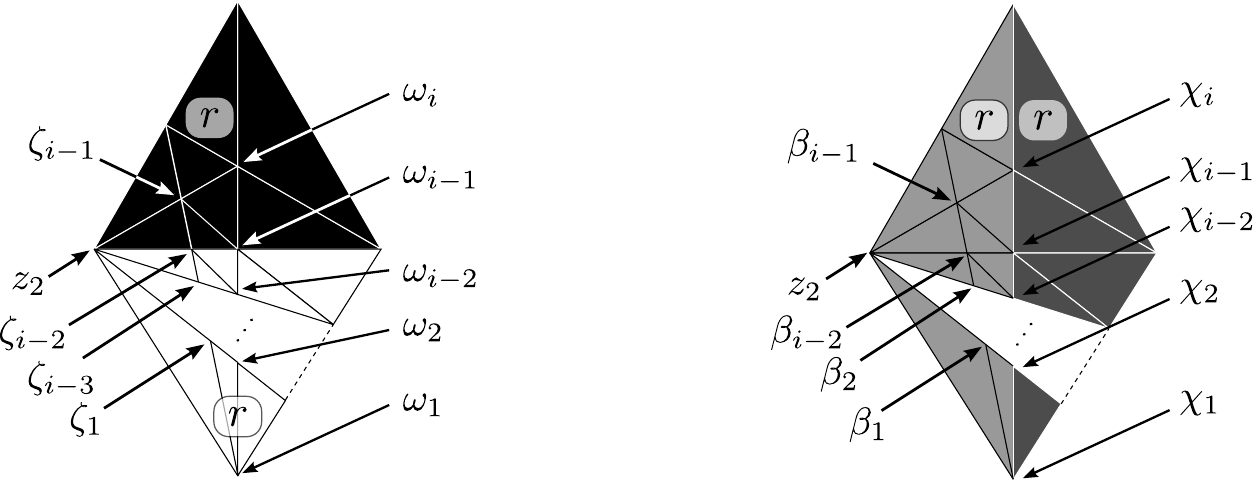}
\caption{Using the PL3-face corresponding to $r$ to define the $\alpha_j$ as $\frac{\zeta_j+\beta_j}{2}$.}
\label{fig:nextdual3}
\end{center}
\end{figure}

The last case is when balloon's tail is refined, that means it is of type $P_i'$ or $B_i'$, $i>1$. We treat the case $B_i'$,
 see Fig. \ref{fig:3d4}. All the 0-simplices $\beta_j$ are already defined, we need to define each
$\alpha_j$ and each $\gamma_j$. Observe that here $r\neq s-1$ and the definitions of $\alpha_j$ and $\gamma_j$ 
are not analogous.

\begin{figure}[!htb] 
\begin{center}
\includegraphics[width=15cm]{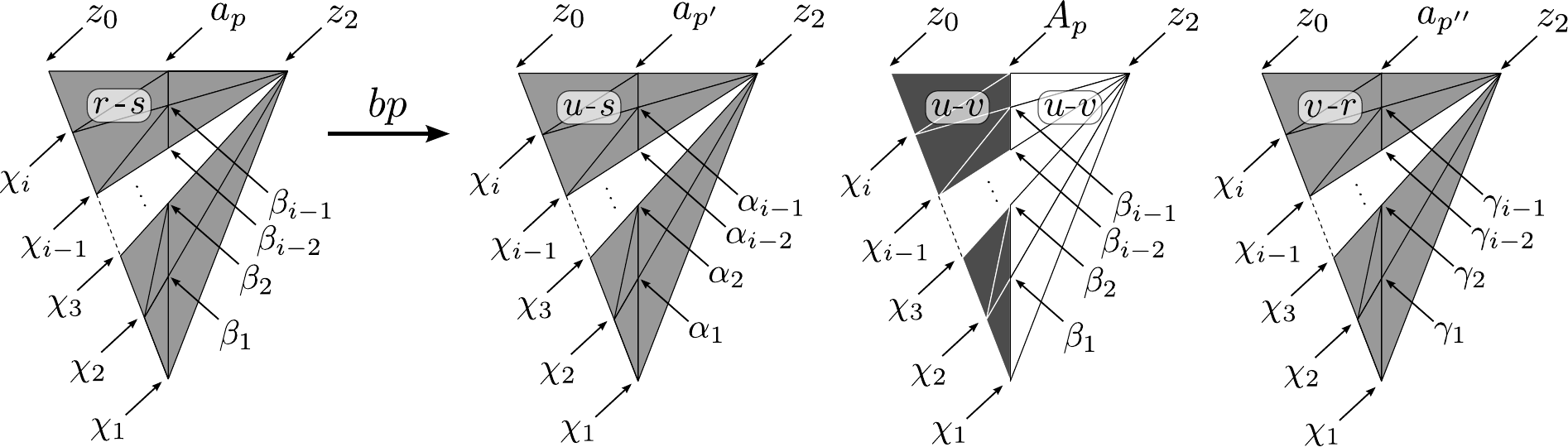}
\caption{Embedding the part of the pillow corresponding to the tail of the balloon: case $B_i'$ of the tail.}
\label{fig:3d4}
\end{center}
\end{figure}

In this case we need to reduce the PL3-faces $\nabla_r$ and $\nabla_s$
to create enough space to build PL2-faces 0- and 2-colored.
To define 0-simplices $\alpha_j$ and $\gamma_j$, one of these cases is analogous to the case not
 refined, but the other we describe here ($\nabla_r$ is in the new case is the rank 
%of
%the PL2$_0$- and PL2$_1$-faces of $\nabla_r$, the new case 
%happens when the rank
 of PL2$_0$-face is equals to the rank of the PL2$_1$-face plus 2, if its not true,
the new case is in the PL3-face $\nabla_v$).
Suppose that the new case is in the PL3-face, $\nabla_r$. To define $\alpha_j$,
suppose that the PL2$_0$-face of this PL3-face is not refined,
 see Fig. \ref{fig:3d5}. Define each $\alpha_j$  as the middle point between $\beta_j$ and $\omega_j$.

\begin{figure}[!htb] 
\begin{center}
\includegraphics[scale=0.7]{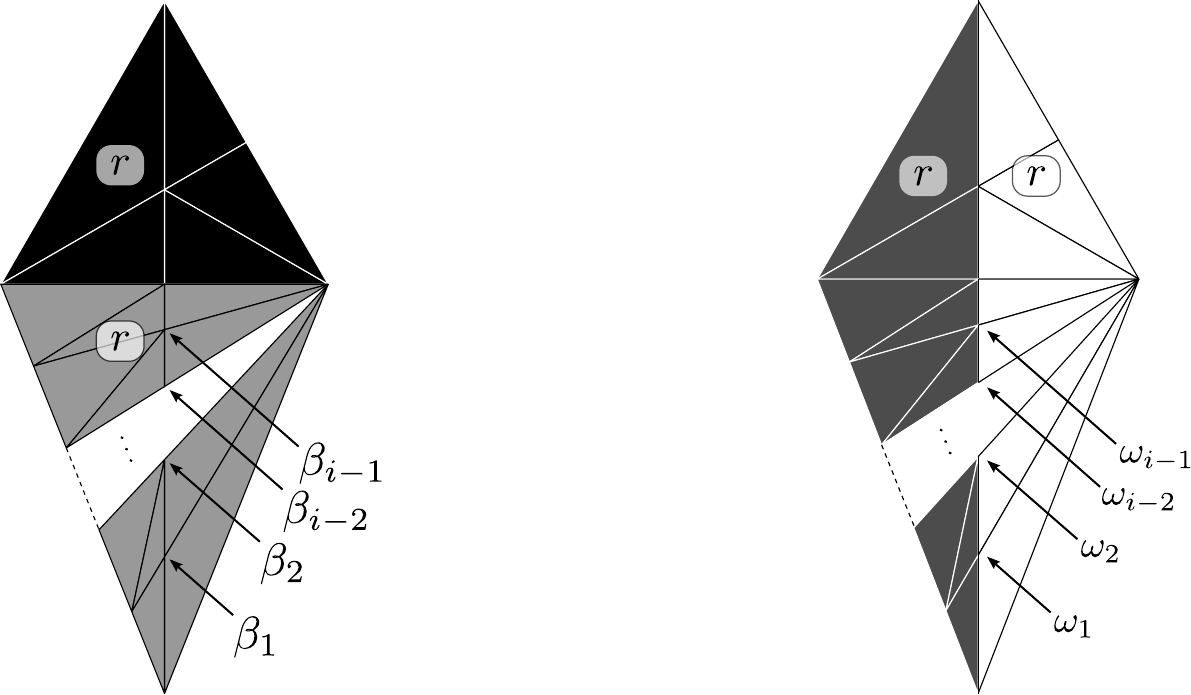}
\caption{Using the PL3-face $\nabla_r$ to define $\alpha_j$ as $\frac{\omega_j+\beta_j}{2}$.}
\label{fig:3d5}
\end{center}
\end{figure}

Consider the case that the PL2$_0$-face, of the PL3-face $\nabla_r$, is refined
see Fig. \ref{fig:3d6}.
 This is a final subtlety which 
is treated with the {\em bump}. \index{bump} This is characterized by a
non-convex pentagon shown in the bottom part Fig. \ref{fig:3d6}.
Let $\nu_j$ be $\frac{z_2+\omega_j}{2}$
and $\alpha_j$ as $\frac{\beta_{j-1}+\nu_j}{2}$, for $j=1,\ldots, i-1$. Observe that
if we define $\alpha_j$ as if the PL2$_0$-face
where not refined, some 1-simplices may cross.

%\begin{figure}[!htb] 
%\begin{center}
%\includegraphics[scale=0.7]{A.figs/3d6.pdf}
%\caption{Using the PL-tetrahedron $\nabla_r$ to define $\alpha_j$ as the middle point between $\mathcal{\nu}_j$ and $\beta_{j-1}$
%(because of the bump, see Fig. \ref{fig:3d7}).}
%The bump: a final dificulty to overcome.}
%\label{fig:3d6}
%\end{center}
%\end{figure}

\begin{figure}[!htb] 
\begin{center}
\includegraphics[scale=0.7]{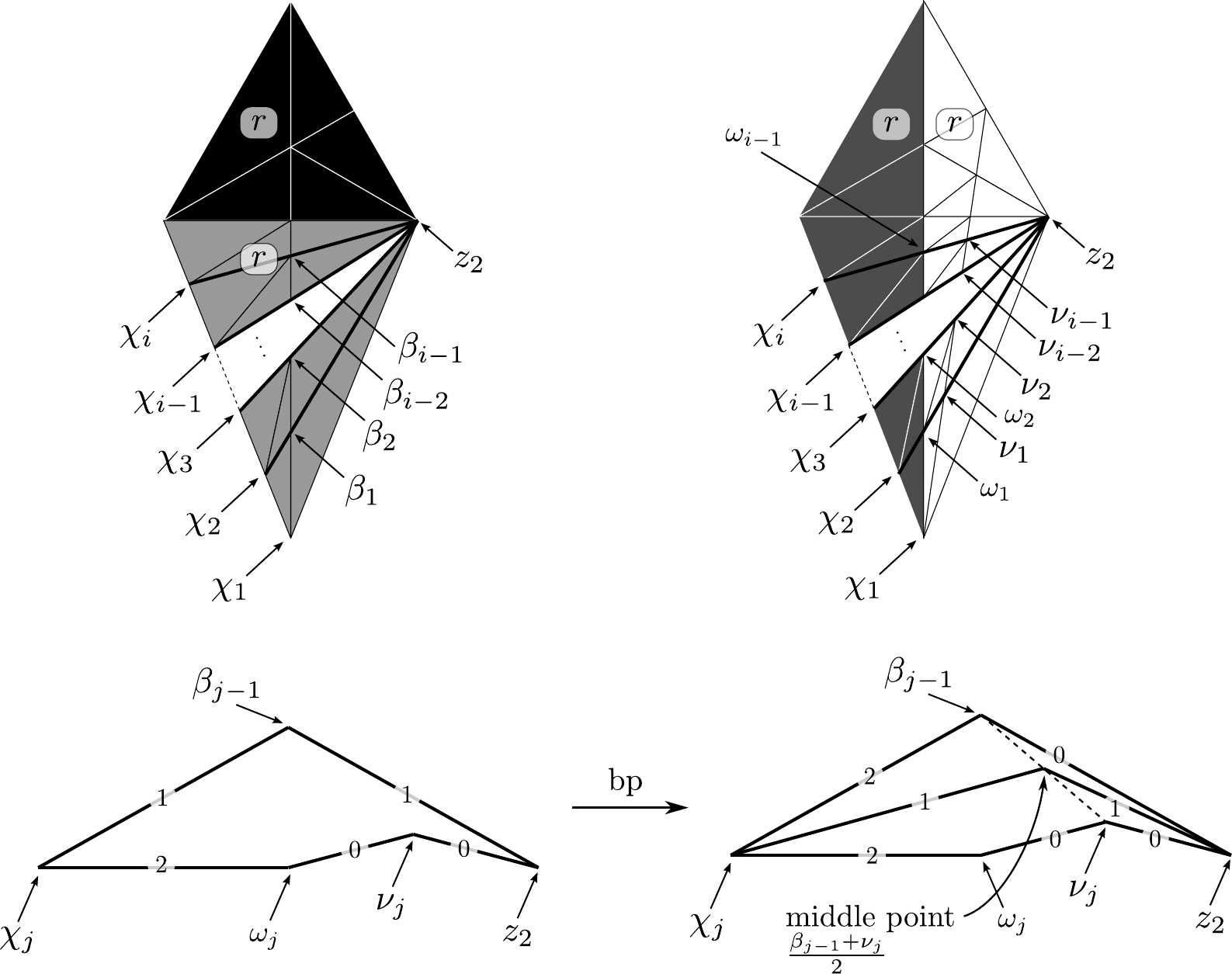}
\caption{The bump: a final subtlety and how to deal with it.}
\label{fig:3d6}
\end{center}
\end{figure}

\end{proof}

\subsection{Obtaining the framed Link}
Given $\mathcal{H}_{n}^\star$ into
$\mathbb{R}^3$ we obtain the $k$-component 
framed link (corresponding to the $k$ twistors) as a set of 
PL-triangulated $k$ cylinders in the 2-skeleton of 
$\mathcal{H}_{n}^\star$, named $\mathcal{C}_1,$ $\mathcal{C}_2, \ldots,\mathcal{C}_k$.
Observe that at this stage every 0-simplex of  $\bigcup_{j=1}^k \mathcal{C}_j$ 
has a 3-D coordinate attached to it. These cylinders are parametrized
as $k$ pairs of isometric rectangles (forming a strip) as in Fig. \ref{fig:strips}.
We draw $k$ straight horizontal lines at different heights of the rectangles, in the example,
lines $c_1-c_1$, $c_2-c_2$ and $c_3-c_3$. These lines are 
mapped into polygons in $\mathbb{R}^3$ which are PL-closed curves. 
The data we need is $\bigcup_{j=1}^k \mathcal{C}_j \subset \mathcal{H}_{n}^\star$
and we can discard the rest of $\mathcal{H}_{n}^\star$. The link that we seek 
is $\bigcup_{j=1}^k c_j$ with framing $\ell_j$, where $\ell_j$ is
given by the linking number of the two components of $\mathcal{C}_j$ oriented in
the same (arbitrary) direction.
We briefly review the definition of linking number 
(\cite {lickorish1997introduction}).
Consider two distinct components $K_1$ and $K_2$ of an oriented link projected into
the plane so that the crossings are transversal 
(no tangency) and that there are no triple points. 
The projection is also {\em decorated}
in the sense that at each crossing the upper and 
the lower strands are given, usually,
by omitting a small segment of the lower strand.
 The \index{linking number} {\em linking number of $\{K_1, K_2\}$
is half of the algebraic sum of the signs of the crossings 
between $K_1$ and $K_2$}, oriented 
in the same direction. If $G$ is a gem, $|G|$ means
the 3-manifold induced by $G$.
The link projection given in Fig. \ref{fig:projecoes2} 
which induce $|r^{24}_5|$ have its three linking numbers $-3$.

%----------------------

\begin{proposition}
\label{prop:totalnumberofsimplices}
 The number of 1-simplices in $\bigcup_{j=1}^k c_j$ is 
at most $12n^2$, where $2n$ is the number
of vertices of the input gem.
\end{proposition}
\begin{proof} 
 A $c_j$ crosses one PL2$_m$-face, for $m\in \{0,1,2,3\}$.
 It is easy to verify that the maximum number of 2-simplices in $B_i'$ or $ P'_i$ is
 $3i-1$ and this number exceeds similar numbers for $B_i$, $ P_i$, $R^b_i$ and $R^p_i$, for 
 $i \ge  1$. The maximum $i$ is $2n-1$, so the maximum number of 2-simplices 0-, 1- and 2-colored in one PL2-face of the 
 final complex is $6n-4$. Each 1-simplex of $c_j$ crosses at most once each 2-simplex.
 Therefore, the number of 1-simpices crossing a 2-simplex 0-, 1- and 2-colored is at most $12n-8$. 
 A $c_j$ crosses at most four 2-simplices 3-colored. The result follows because $n$ is an upper 
 bound for $k$, number of components. Just note that a component in the link
 is in 1-1 correspondence with the twistors of the original gem and a twistor 
is formed by 2 vertices. This proof is partially ilustrated in Fig. \ref{fig:strips}.
 The illustration is not faithful because we can replace the strips at the right by
their bottom parts, getting simpler cylinders homotopic to the ones illustrated.
\end{proof}

Each cylinder is formed by two strips. Each strip by two adequate
pairs of two PL2-faces in the boundary of the PL3-faces in the hinge.
In the way depicted we have one PL3-face of each color $c$, $c=0,1,2,3$.
The number of  crossings of the curves $c_j$, $j=1,2,3$ coincides with the
number of $1$-simplices (or $0$-simplices) of $\bigcup_{j=1}^3 \mathcal{C}_j$.
To decrease this number we can replace the part of the strip which does not use a
PL2$_3$-face by the two complementary PL2$_3$-faces in the corresponding PL3-face.
This produces isotopic cylinders, but the number of crossings of the $c_j$'s are smaller. 
Note that each PL2$_3$-face has just five 2-simplices. In Fig. \ref{fig:strips} we 
depict the situation for the wings arising from $r_5^{24}$ before the 3 replacements.

\begin{figure}[!htb]
\begin{center}
\includegraphics[width=15cm]{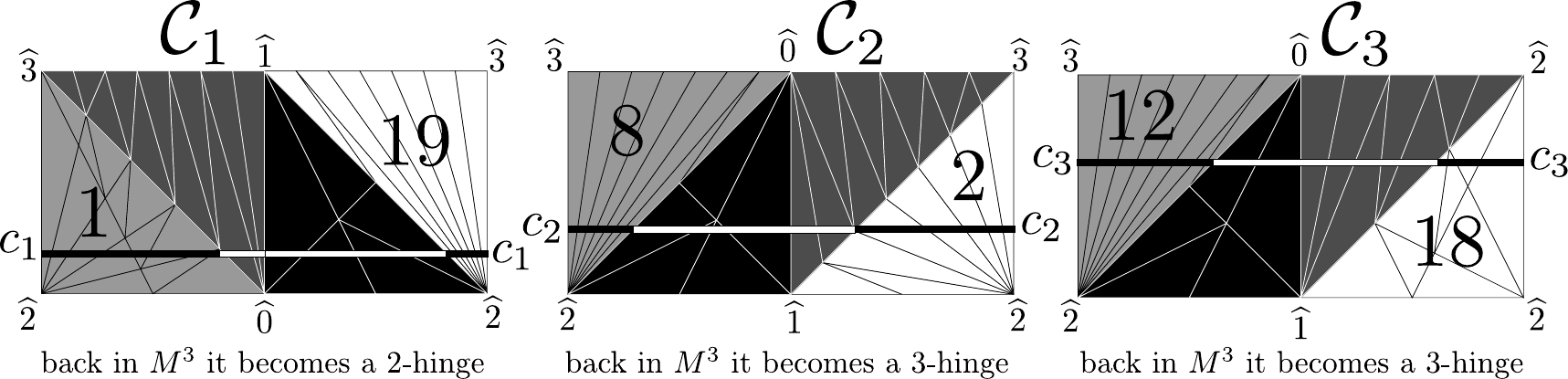}
\caption{From hinges to cylinders 
$\mathcal{C}_j$  to curves  $c_j$ (example inducing  $|r^{24}_5|$).
%Each cylinder is formed by two strips. Each strip by two adequate
%pairs of two PL2-faces in the boundary of the PL3-faces in the hinge.
%In the way depicted we have one PL3-face of each color $c$, $c=0,1,2,3$.
%The number of  crossings of the curves $c_j$, $j=1,2,3$ coincides with the
%number of $1$-simplices (or $0$-simplices) of $\bigcup_{j=1}^3 \mathcal{C}_j$.
%To decrease this number we can replace the right part of each strip by the
%two complementary faces at each right PL3-face. This produce isotopic cylinders,
%but the number of crossings of the $c_j$'s are smaller.
}  
\label{fig:strips}
\end{center}
\end{figure}

%----------------------
\begin{figure}[!htb]
\begin{center}
\includegraphics[scale=0.9]{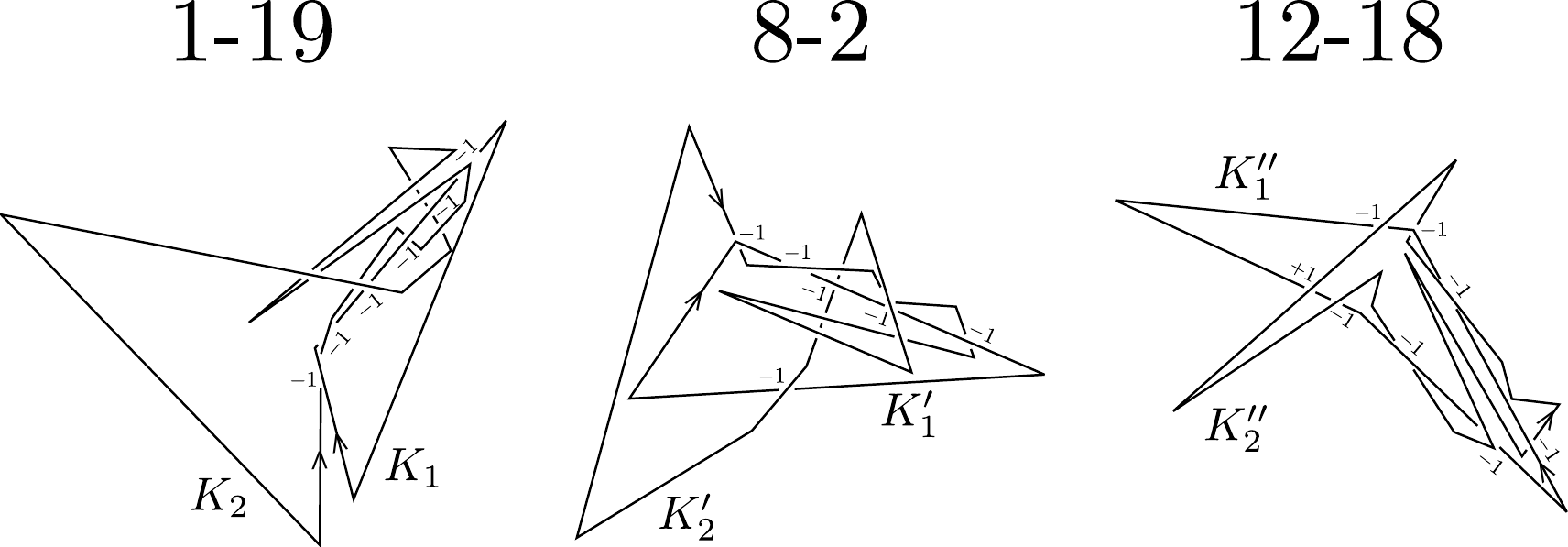}
\caption{Projection of the algorithm's output, yielding 
the linking numbers of the 
boundary components of the embedded cylinders of Fig. \ref{fig:strips}.}
\label{fig:projecoes2}
\end{center}
\end{figure}

\subsection{Obtaining a Gauss code for the link}

At this point, we have the link as a set of cyclic sequences of points in $\mathbb{R}^3$.
We also have the framing of each component. Thus the theoretical problem is solved. 
However, is convenient to go on getting adequate planar projections to produce 
planar diagram for the link.
We obtain the following Gauss code, 
(\cite{rosenstiehl1976solution}, chapter 3 of \cite{Lins1980}, 
\cite{lins2008hsg}),
where signs mean
up $(+)$ and down $(-)$ passages:
$ ((-2,+3, -4,+1), (-5, +6, +2, -1),$ $(-3, +4, -7, -6, +5, +7)).$
From this code we get the link planar diagram of 
Figs. \ref{fig:r24_5link2}.
Since we have the framings curls
can be removed. An explicit elegant framed link inducing the euclidean 
3-manifold $|r^{24}_5|$ was previously unknown. We use 
Fig. \ref{fig:r24_5link} as input for L. Lins's software \cite{lins2007blink} to obtain
the WRT-invariants from $r=3$ to $r=20$ for the space $|r^{24}_5|$. We also apply our algorithm
for the Weber-Seifert hyperbolic dodecahedron space, obtaining a link with
142 crossings, included in Appendix B.

\begin{figure}[!htb]
\begin{center}
\includegraphics[width=13cm]{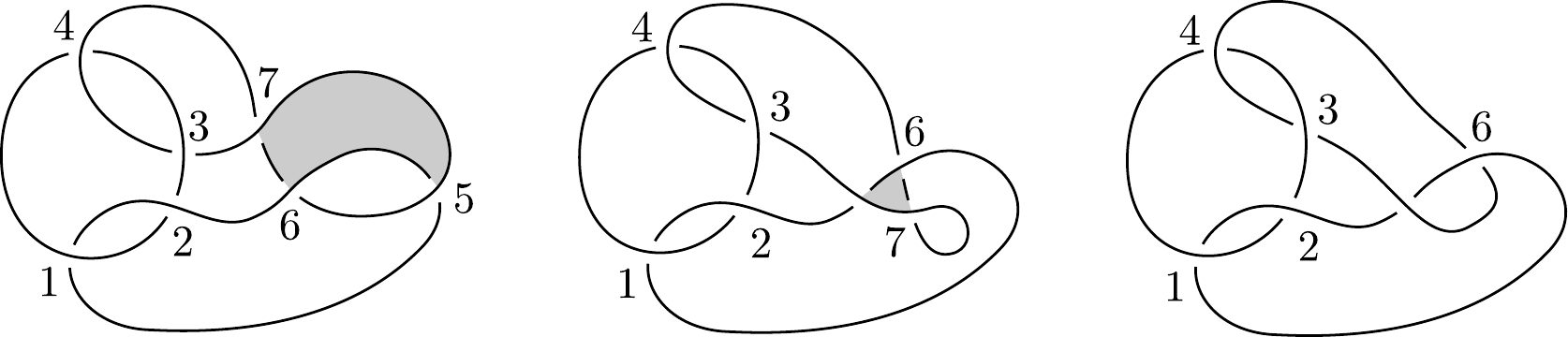}
\caption{Simplifying the link diagram for $|r_5^{24}|$.}
\label{fig:r24_5link2}
\end{center}
\end{figure}

\begin{figure}[!htb]
\begin{center}
\includegraphics[width=13cm]{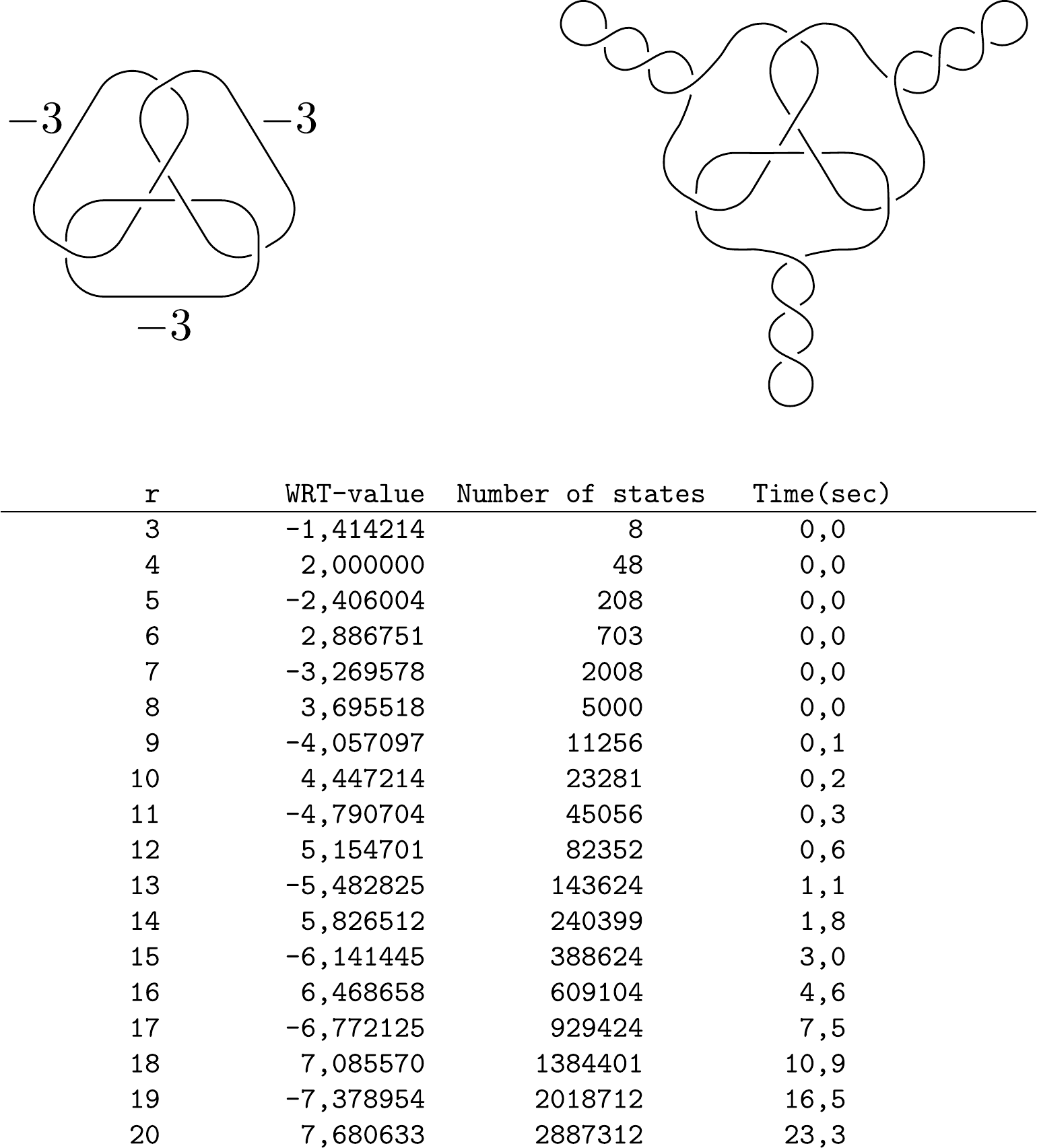}
\caption{Framed link, blackboard framed link and WRT-invariants for $|r_5^{24}|$.
 %These invariants assume real values because the space is symmetric: the two
%oriented forms agree.
 Data obtained from L. Lins software \cite{lins2007blink}.}
\label{fig:r24_5link}
\end{center}
\end{figure}

%\begin{figure}[!htb]
%\begin{center}
%\includegraphics[width=13cm]{A.figs/r24_5linkinvariants.pdf}
%\caption{ WRT-invariants for $|r_5^{24}|$.
% These invariants assume real values because the space is symmetric: the two
%oriented forms agree. Data obtained from L. Lins software \cite{lins2007blink}.}
%\label{fig:r24_5link}
%\end{center}
%\end{figure}

\newpage
\begin{algorithm}
\label{theo:algorithm}
 There exists an $O(n^2)$-algorithm to produce, from a
  resoluble gem inducing an $M^3$,  a blackboard framed link also inducing $M^3$.
\end{algorithm}
\begin{proof} 
We start with a resoluble gem $G$ with $2n$ vertices. Here is the algorithm,
justified by the theory previously developed:
\begin{itemize}
 \item Form decreasing sequence of gems starting with $\mathcal{J}^2$,
 the $J^2$-gem associated with
the resolution of $G$ performing adequate $0$- or $1$-flips
 and finishing at the bloboid $\mathcal{B}_1$:
$\mathcal{J}^2=\mathcal{H}_{n},  \mathcal{H}_{n-1}, \ldots, \mathcal{H}_1=\mathcal{B}_1$.
%See Fig. \ref{fig:j2gemrSEQ} for the $r^{24}_5$-example.
 
 \item Form sequence of balloons $\mathcal{B}_1^\star,\ldots, \mathcal{B}_{n-1}^\star$
and pillows $\mathcal{P}_2^\star,\ldots, \mathcal{P}_n^\star$ defining implicitly the 
  sequence of combinatorial $2$-complexes $\mathcal{H}_1^\star, \mathcal{H}_1^\star, 
\ldots, \mathcal{H}_{2n}^\star$, together with 
their respective wings and nervures (combinatorially   
given by rotations), 
$\mathcal{W}_1^\ell \cup \mathcal{N}_1^\ell$, 
$\mathcal{W}_2^\ell \cup \mathcal{N}_2^\ell$,
\ldots, $\mathcal{W}_n^\ell \cup \mathcal{N}_n^\ell$ and 
$\mathcal{W}_1^r \cup \mathcal{N}_1^r$, 
$\mathcal{W}_2^r \cup \mathcal{N}_2^r$,
\ldots, $\mathcal{W}_n^r \cup \mathcal{N}_n^r$.
Each $bp$-move is simply a flip in the
primal sequence. See Figs. \ref{fig:winglist01} to \ref{fig:winglist11}
from the $r^{24}_5$-example.

\item Use Tutte's barycentric method with the edge weight 
heuristic (Fig. \ref{fig:arvorecompesos} from the $r^{24}_5$-example) 
to provide rectilinear embeddings
of $\mathcal{W}_n^\ell \cup \mathcal{N}_n^\ell$ in $\Pi_\ell$ and
of $\mathcal{W}_n^r \cup \mathcal{N}_n^r$ in $\Pi_r$ fixing the outer regions.
The nervures
are useful up to this point, and after obtaining the rectilinear embeddings
they can be discarded.

\item Get $\mathcal{H}_{1}^\diamond$ using
$\mathcal{W}_n^\ell \cup \mathcal{W}_n^r$ 
by the cone construction.

\item Get the sequence $\mathcal{H}_{2}^\diamond, \ldots, \mathcal{H}_{n}^\diamond = \mathcal{H}_n^\star$
by the blowing up technique in the proof of Theorem \ref{theo:teoremadeumalinha}.

\item Define the framings of the components of the link
as the linking numbers of the boundaries of the cylinders formed by
the strips coming from the hinges; special care: distinguish the
hinges which become 2-hinges from the
hinges which become 3-hinges in $M^3$. See Fig. \ref{fig:strips}.

\item Find an adequate projection of suitable medial curves in the cylinders.
These curves form the link. Find Gauss code for the link, and so,
a projection is combinatorially specified (\cite{rosenstiehl1976solution}, 
chapter 3 of \cite{Lins1980}, 
\cite{lins2008hsg}). Add curls to produce
a blackboard framed link. See Figs. \ref{fig:r24_5link2} and \ref{fig:r24_5link}
from the $r^{24}_5$-example.
\end{itemize}
This algorithm has both space complexity and time complexity $O(n^2)$.
Its output, first obtained as a set with no more than $n$ 
PL-polygons in $\mathbb{R}^3$,  
has a total of at most $12n^2$ vertices.
\end{proof}

%\chapter{Table of QI of spaces EUCLID$_i$}

%\section{EUCLID$_0$: Rep. $r_1^{24}$}

%\begin{figure}[!htb]
%\begin{center}
%\includegraphics[width=10.5cm]{A.figs/r24-1.pdf}
%\caption{3-gem, blackboard framed link for $|r_1^{24}|$ and
%WRT-invariants.}
%\label{fig:r24_5link}
%\end{center}
%\end{figure}

%\newpage

%\section{EUCLID$_1$: Rep. $r_5^{24}$}

%\begin{figure}[!htb]
%\begin{center}
%\includegraphics[width=10.5cm]{A.figs/r24_5GemLinkWRT.pdf}
%\caption{3-gem, blackboard framed link for $|r_5^{24}|$ and
%WRT-invariants.}
%\label{fig:r24_5link}
%\end{center}
%\end{figure}

%\newpage

%\section{EUCLID$_3$: Rep. $r_6^{24}$}%

%\begin{figure}[!htb]
%\begin{center}
%\includegraphics[width=10.5cm]{A.figs/r24_6.pdf}
%\caption{3-gem, blackboard framed link for $|r_6^{24}|$ and
%WRT-invariants.}
%\label{fig:r24_5link}
%\end{center}
%\end{figure}

%-----------------------------------

%\bibliographystyle{is-alpha}
%\addcontentsline{toc}{bibliografia}{\MakeTextUppercase{Referências Bibliográficas}}
%\bibliography{d:/slsl\3.DadosSostenes.35.ArtigosLivros.bibtexGoogleScholar/bibtexIndex.bib} % bib file is slsl.bib
%\bibliography{~/home/ricardo/Dropbox/35.ArtigosLivros.bibtexGoogleScholar/bibtexIndex.bib}
%\bibliography{bibtexIndex.bib}
%\bibliography{slsl}

% \printindex
-----------

%\section{Appendix A: Proofs}\label{appA}

%\begin{proposition}
%\label{prop:planetwist}
%Let $G$ be a gem. $(a)$ A $ji$-twisting of a $j$-twistor of $G$ is factorable as
%one $i$-flip and one $j$-flip.
%$(b)$ If $G$ has a $0$-consecutive labelling on its vertices,
%then a $ji$-twisting can be accomplished by one $k$-flip (which maintains planarity
%of the $\widehat{0}$-residue) followed
%by the $\{u,v\}$ label interchange. The final gem has a $0$-consecutive labelling
%and so the $ji$-twisting is entirely depicted in the $\widehat{0}$-residue,
%a plane graph.
%\end{proposition}

%\begin{proof}\hspace{-4mm} ({\bf Theorem \ref{theo:bound}})
%The proof is provided at the end of the paper, as Algorithm
%\ref{theo:algorithm}.
%\end{proof}

\section*{Appendix A: a solution for the 
Weber-Seifert Dodecahedral Hyperbolic Space}

We found a projection for a framed link with 142 crossings for the 
Weber-Seifert Dodecahedral Hyperbolic Space. Actually the PL-link are nine PL-polygons in
$\mathbb{R}^3$ with a total of only 68 vertices.
The data that follows is a positive answer for Jeffrey Weeks' question more than
twenty years ago. Still in raw form, it can be substantially simplified

We apply our algorithm to the 50-vertex gem 
and its resolution given at the right side of Fig. 
\ref{fig:resolutionDhip50A}. 

%-----------------------------------
\begin{figure}[!htb]
\begin{center}
\includegraphics[width=14.5cm]{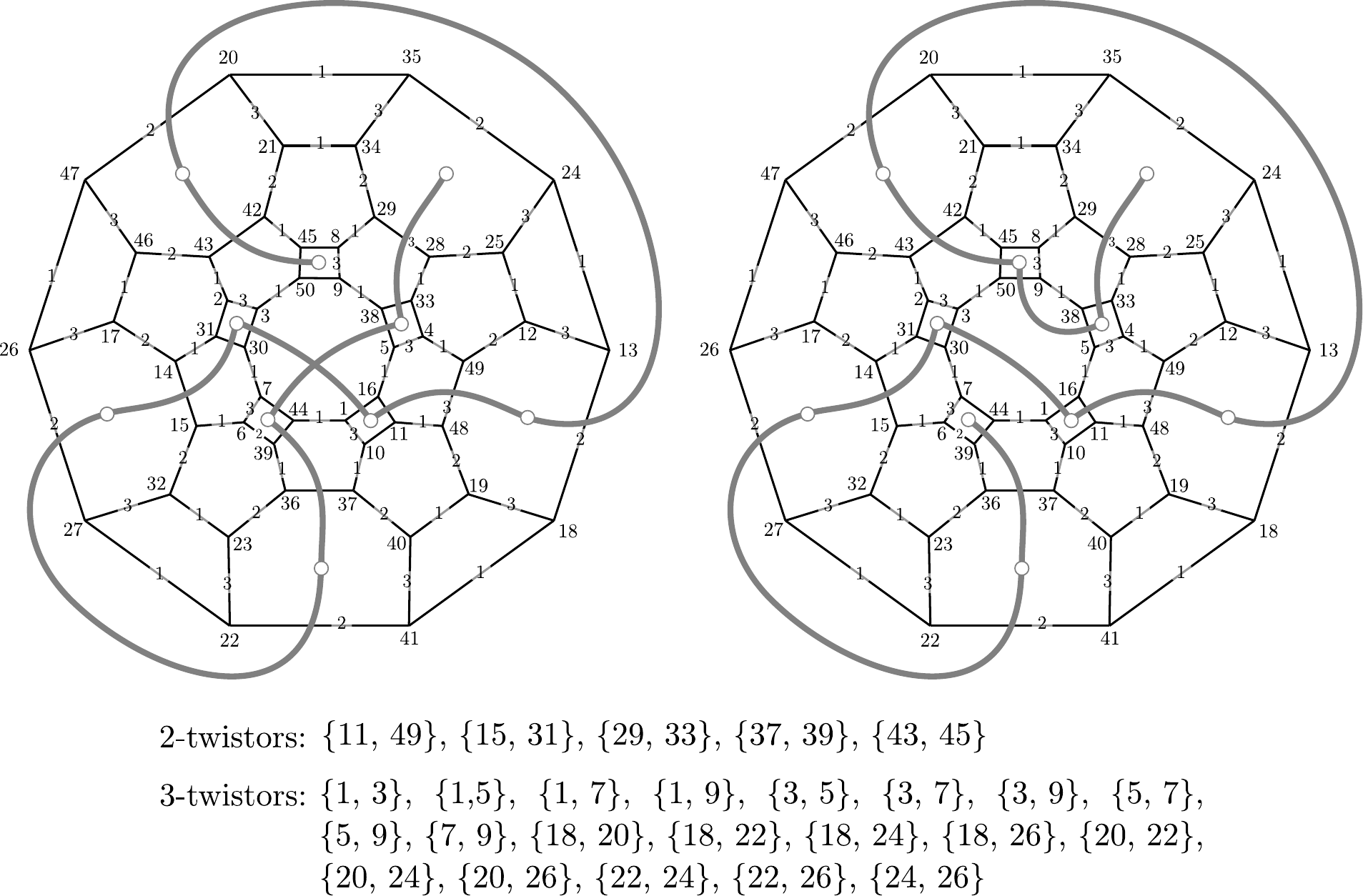} \\
\caption{\sf This is a 50-vertex gem which behaves as the {\em attractor}.}
%(see \cite{lins1995gca})
%for the Weber-Seifert dodecahedral hyperbolic space.
%For a proof that it induces this space see also \cite{lins1995gca}.} 
\label{fig:resolutionDhip50A}
\end{center}
\end{figure}
%-----------------------------------

A crossing $x\in \mathbb{N}$ has 4 legs in counterclockwise 
order: $(4x-3, 4x-2, 4x-1, 4x)$.
A {\em duet} \index{duet} is a perfect matching of the legs.
The first entry of a \index{quintet} {\em quintet} is the number of the crossing.
Each crossing appears in two consecutive quintets. The second entry is $d$ 
or $u$ depending on whether the quintet holds the first or the second occurrence
of its crossing. The $d$ means that the southwest to the northeast passage
goes under, the $u$ means that it goes over. The third and fourth entries of a quintet
are legs and their order specifies a consistent orientation for all the components
of the link. The fifth and last entry of a quintet is the number of the component
of the link that contains the two legs. By properly embedding the quintets in the plane
and identifying the legs as specified by the duets we have a link diagram
with consistent orientation of all of its component. Thus to obtain a Gauss code
(\cite{rosenstiehl1976solution}, chapter 3 of \cite{Lins1980}, 
\cite{lins2008hsg}), for the link is straightforward. 
Even though 
there are 142 crossings in the projection, the number of 1-simplices
in the PL-link is only 68. This was obtained by a shortcutting technique 
which started with over two hundred 1-simplices:
each 0-simplex defines a triangle in $\mathbb{R}^3$; if this triangle
is not pierced by a 1-simplex, then the 0-simplex is removed from the link. 
Compare 68 with our theoretical 
bound, namely $12n^2=7500$, since $n=25$. We emphasize the issue that
the algorithm behaves very efficiently.

\subsection{Duets of $DHIP_{50}^{142}$}
{\scriptsize
\begin{center}
$
\begin{array}{|c|c|c|c|c|c|c|c|} \hline
 1, 8 &
 2, 155 &
 3, 112 &
 4, 111 &
 5, 108 &
 6, 11 &
 7, 104 &
 9, 16 \\
 10, 101&
 12, 107&
 13, 564&
 14, 19&
 15, 102&
 17, 24&
 18, 97&
 20, 159 \\
 21, 158&
 22, 27&
 23, 468&
 25, 32&
 26, 465&
 28, 163&
 29, 340&
 30, 35 \\
 31, 336&
 33, 40&
 34, 267&
 36, 271&
 37, 560&
 38, 43&
 39, 268&
 41, 48 \\
 42, 75&
 44, 559&
 45, 72&
 46, 51&
 47, 76&
 49, 56&
 50, 495&
 52, 71 \\
 53, 68&
 54, 59&
 55, 188&
 57, 62&
 58, 185&
 60, 63&
 61, 500&
 64, 67 \\
 65, 70&
 66, 183&
 69, 184&
 73, 494&
 74, 79&
 77, 84&
 78, 263&
 80, 259 \\
 81, 258&
 82, 85&
 83, 26&
 86, 475&
 87, 92&
 88, 335&
 89, 472&
 90, 93 \\
 91, 466&
 94, 471&
 95, 98&
 96, 467&
 99, 156&
 100, 103&
 105, 110&
 106, 319 \\
 109, 320&
 113, 240&
 114, 239&
 115, 526&
 116, 117&
 118, 525&
 119, 122&
 120, 235 \\
 121, 236&
 123, 142&
 124, 127&
 125, 132&
 126, 531&
 128, 139&
 129, 136&
 130, 135 \\
 131, 532&
 133, 138&
 134, 317&
 137, 230&
 140, 141&
 143, 146&
 144, 231&
 145, 524 \\
 147, 528&
 148, 149&
 150, 527&
 151, 154&
 152, 383&
 153, 244&
 157, 164&
 160, 563 \\
 161, 464&
 162, 167&
 165, 172&
 166, 337&
 168, 463&
 169, 344&
 170, 175&
 171, 338 \\
 173, 180&
 174, 269&
 176, 275&
 177, 556&
 178, 181&
 179, 558&
 182, 497&
 186, 191 \\
 187, 496&
 189, 196&
 190, 489&
 192, 427&
 193, 432&
 194, 199&
 195, 488&
 197, 204 \\
 198, 435&
 200, 431&
 201, 448&
 202, 207&
 203, 544&
 205, 212&
 206, 443&
 208, 447 \\
 209, 446&
 210, 215&
 211, 538&
 213, 220&
 214, 483&
 216, 371&
 217, 370&
 218, 223 \\
 219, 376&
 221, 226&
 222, 373&
 224, 311&
 225, 522&
 227, 316&
 228, 229&
 232, 523 \\
 233, 530&
 234, 237&
 238, 529&
 241, 470&
 242, 247&
 243, 324&
 245, 252&
 246, 321 \\
 248, 327&
 249, 326&
 250, 253&
 251, 568&
 254, 329&
 255, 260&
 256, 479&
 257, 332 \\
 261, 334&
 262, 265&
 266, 333&
 270, 557&
 272, 339&
 273, 280&
 274, 553&
 276, 507 \\
 277, 506&
 278, 283&
 279, 416&
 281, 288&
 282, 413&
 284, 411&
 285, 552&
 286, 291 \\
 287, 504&
 289, 296&
 290, 295&
 292, 359&
 293, 364&
 294, 299&
 297, 304&
 298, 365 \\
 300, 363&
 301, 452&
 302, 305&
 303, 366&
 306, 451&
 307, 312&
 308, 369&
 309, 396 \\
 310, 313&
 314, 395&
 315, 318&
 322, 567&
 323, 392&
 325, 330&
 328, 473&
 331, 476 \\
 341, 404&
 342, 347&
 343, 408&
 345, 352&
 346, 511&
 348, 403&
 349, 516&
 350, 353 \\
 351, 512&
 354, 515&
 355, 360&
 356, 551&
 357, 514&
 358, 361&
 362, 519&
 367, 372 \\
 368, 445&
 374, 379&
 375, 484&
 377, 382&
 378, 535&
 380, 521&
 381, 536&
 384, 387 \\
 385, 390&
 386, 565&
 388, 391&
 389, 566&
 393, 456&
 394, 399&
 397, 402&
 398, 459 \\
 400, 455&
 401, 460&
 405, 510&
 406, 409&
 407, 508&
 410, 549&
 412, 505&
 414, 503 \\
 415, 420&
 417, 502&
 418, 421&
 419, 554&
 422, 501&
 423, 428&
 424, 499&
 425, 430 \\
 426, 429&
 433, 440&
 434, 485&
 436, 543&
 437, 542&
 438, 441&
 439, 546&
 442, 541 \\
 444, 537&
 449, 518&
 450, 453&
 454, 517&
 457, 462&
 458, 561&
 461, 562&
 469, 474 \\
 477, 482&
 478, 533&
 480, 539&
 481, 534&
 486, 545&
 487, 490&
 491, 548&
 492, 493 \\
 498, 555&
 509, 550&
 513, 520&
 540, 547 & & & & \\ \hline
\end{array}
$
\end{center}
}

%\subsection{Cylinders and framings}
%\begin{verbatim}
%boundary cylinder 1 linking number from 2-twistor {19,13} is -1
%boundary cylinder 2 linking number from 2-twistor {3,29} is 0
%boundary cylinder 3 linking number from 2-twistor {49,33} is -1
%boundary cylinder 4 linking number from 2-twistor {39,43} is -1
%boundary cylinder 5 linking number from 2-twistor {9,23} is -2
%boundary cylinder 6 linking number from 3-twistor {1,31} is 0
%boundary cylinder 7 linking number from 3-twistor {6,26} is 0
%boundary cylinder 8 linking number from 3-twistor {46,36} is 0
%boundary cylinder 9 linking number from 3-twistor {11,21} is 0
%\end{verbatim}

\subsection{Quintets, cylinders and framings of $DHIP_{50}^{142}$}

{\tiny
\begin{center}
$
\begin{array}{cccccc}
 \{1,\text{d},3,1,1\} & \{1,\text{u},4,2,3\} &
 \{2,\text{d},7,5,8\} & \{2,\text{u},8,6,1\} &
 \{3,\text{d},11,9,1\} & \{3,\text{u},10,12,1\} \\
 \{4,\text{d},13,15,8\} & \{4,\text{u},16,14,1\} &
 \{5,\text{d},18,20,2\} & \{5,\text{u},19,17,1\} &
 \{6,\text{d},24,22,1\} & \{6,\text{u},21,23,9\} \\
 \{7,\text{d},27,25,1\} & \{7,\text{u},28,26,5\} &
 \{8,\text{d},32,30,1\} & \{8,\text{u},31,29,4\} &
 \{9,\text{d},35,33,1\} & \{9,\text{u},34,36,3\} \\
 \{10,\text{d},39,37,9\} & \{10,\text{u},40,38,1\} &
 \{11,\text{d},42,44,8\} & \{11,\text{u},43,41,1\} &
 \{12,\text{d},48,46,1\} & \{12,\text{u},45,47,1\} \\
 \{13,\text{d},51,49,1\} & \{13,\text{u},50,52,6\} &
 \{14,\text{d},53,55,2\} & \{14,\text{u},56,54,1\} &
 \{15,\text{d},59,57,1\} & \{15,\text{u},60,58,8\} \\
 \{16,\text{d},61,63,8\} & \{16,\text{u},62,64,1\} &
 \{17,\text{d},66,68,2\} & \{17,\text{u},67,65,1\} &
 \{18,\text{d},70,72,1\} & \{18,\text{u},71,69,6\} \\
 \{19,\text{d},73,75,8\} & \{19,\text{u},76,74,1\} &
 \{20,\text{d},80,78,4\} & \{20,\text{u},79,77,1\} &
 \{21,\text{d},84,82,1\} & \{21,\text{u},81,83,3\} \\
 \{22,\text{d},86,88,9\} & \{22,\text{u},85,87,1\} &
 \{23,\text{d},91,89,9\} & \{23,\text{u},92,90,1\} &
 \{24,\text{d},93,95,1\} & \{24,\text{u},96,94,5\} \\
 \{25,\text{d},99,97,2\} & \{25,\text{u},98,100,1\} &
 \{26,\text{d},102,104,8\} & \{26,\text{u},103,101,1\} &
 \{27,\text{d},108,106,8\} & \{27,\text{u},107,105,1\} \\
 \{28,\text{d},109,111,3\} & \{28,\text{u},110,112,1\} &
 \{29,\text{d},115,113,6\} & \{29,\text{u},114,116,2\} &
 \{30,\text{d},120,118,8\} & \{30,\text{u},117,119,2\} \\
 \{31,\text{d},123,121,2\} & \{31,\text{u},122,124,2\} &
 \{32,\text{d},126,128,6\} & \{32,\text{u},127,125,2\} &
 \{33,\text{d},129,131,8\} & \{33,\text{u},132,130,2\} \\
 \{34,\text{d},135,133,2\} & \{34,\text{u},134,136,8\} &
 \{35,\text{d},138,140,2\} & \{35,\text{u},139,137,6\} &
 \{36,\text{d},141,143,2\} & \{36,\text{u},144,142,2\} \\
 \{37,\text{d},146,148,2\} & \{37,\text{u},145,147,6\} &
 \{38,\text{d},149,151,2\} & \{38,\text{u},150,152,8\} &
 \{39,\text{d},154,156,2\} & \{39,\text{u},155,153,3\} \\
 \{40,\text{d},159,157,2\} & \{40,\text{u},160,158,9\} &
 \{41,\text{d},164,162,2\} & \{41,\text{u},161,163,5\} &
 \{42,\text{d},167,165,2\} & \{42,\text{u},166,168,8\} \\
 \{43,\text{d},172,170,2\} & \{43,\text{u},171,169,4\} &
 \{44,\text{d},175,173,2\} & \{44,\text{u},174,176,3\} &
 \{45,\text{d},179,177,9\} & \{45,\text{u},180,178,2\} \\
 \{46,\text{d},181,183,2\} & \{46,\text{u},184,182,6\} &
 \{47,\text{d},188,186,2\} & \{47,\text{u},185,187,8\} &
 \{48,\text{d},192,190,7\} & \{48,\text{u},191,189,2\} \\
 \{49,\text{d},196,194,2\} & \{49,\text{u},195,193,7\} &
 \{50,\text{d},200,198,5\} & \{50,\text{u},199,197,2\} &
 \{51,\text{d},204,202,2\} & \{51,\text{u},203,201,7\} \\
 \{52,\text{d},207,205,2\} & \{52,\text{u},206,208,5\} &
 \{53,\text{d},209,211,7\} & \{53,\text{u},212,210,2\} &
 \{54,\text{d},214,216,8\} & \{54,\text{u},215,213,2\} \\
 \{55,\text{d},217,219,4\} & \{55,\text{u},220,218,2\} &
 \{56,\text{d},222,224,5\} & \{56,\text{u},223,221,2\} &
 \{57,\text{d},226,228,2\} & \{57,\text{u},227,225,6\} \\
 \{58,\text{d},230,232,6\} & \{58,\text{u},229,231,2\} &
 \{59,\text{d},233,235,8\} & \{59,\text{u},236,234,2\} &
 \{60,\text{d},237,239,2\} & \{60,\text{u},240,238,6\} \\
 \{61,\text{d},244,242,3\} & \{61,\text{u},241,243,9\} &
 \{62,\text{d},247,245,3\} & \{62,\text{u},246,248,4\} &
 \{63,\text{d},251,249,9\} & \{63,\text{u},252,250,3\} \\
 \{64,\text{d},253,255,3\} & \{64,\text{u},254,256,5\} &
 \{65,\text{d},257,259,4\} & \{65,\text{u},260,258,3\} &
 \{66,\text{d},263,261,4\} & \{66,\text{u},264,262,3\} \\
 \{67,\text{d},266,268,9\} & \{67,\text{u},265,267,3\} &
 \{68,\text{d},270,272,8\} & \{68,\text{u},271,269,3\} &
 \{69,\text{d},276,274,8\} & \{69,\text{u},275,273,3\} \\
 \{70,\text{d},279,277,6\} & \{70,\text{u},280,278,3\} &
 \{71,\text{d},284,282,5\} & \{71,\text{u},283,281,3\} &
 \{72,\text{d},287,285,9\} & \{72,\text{u},288,286,3\} \\
 \{73,\text{d},292,290,6\} & \{73,\text{u},291,289,3\} &
 \{74,\text{d},296,294,3\} & \{74,\text{u},295,293,6\} &
 \{75,\text{d},299,297,3\} & \{75,\text{u},300,298,4\} \\
 \{76,\text{d},304,302,3\} & \{76,\text{u},303,301,5\} &
 \{77,\text{d},305,307,3\} & \{77,\text{u},308,306,8\} &
 \{78,\text{d},312,310,3\} & \{78,\text{u},311,309,5\} \\
 \{79,\text{d},313,315,3\} & \{79,\text{u},314,316,6\} &
 \{80,\text{d},319,317,8\} & \{80,\text{u},318,320,3\} &
 \{81,\text{d},324,322,9\} & \{81,\text{u},323,321,4\} \\
 \{82,\text{d},326,328,9\} & \{82,\text{u},327,325,4\} &
 \{83,\text{d},330,332,4\} & \{83,\text{u},331,329,5\} &
 \{84,\text{d},335,333,9\} & \{84,\text{u},334,336,4\} \\
 \{85,\text{d},339,337,8\} & \{85,\text{u},340,338,4\} &
 \{86,\text{d},344,342,4\} & \{86,\text{u},341,343,5\} &
 \{87,\text{d},347,345,4\} & \{87,\text{u},346,348,9\} \\
 \{88,\text{d},351,349,6\} & \{88,\text{u},352,350,4\} &
 \{89,\text{d},353,355,4\} & \{89,\text{u},354,356,8\} &
 \{90,\text{d},357,359,6\} & \{90,\text{u},360,358,4\} \\
 \{91,\text{d},364,362,6\} & \{91,\text{u},361,363,4\} &
 \{92,\text{d},365,367,4\} & \{92,\text{u},368,366,5\} &
 \{93,\text{d},372,370,4\} & \{93,\text{u},371,369,8\} \\
 \{94,\text{d},376,374,4\} & \{94,\text{u},375,373,5\} &
 \{95,\text{d},379,377,4\} & \{95,\text{u},380,378,6\} &
 \{96,\text{d},382,384,4\} & \{96,\text{u},383,381,8\} \\
 \{97,\text{d},387,385,4\} & \{97,\text{u},386,388,9\} &
 \{98,\text{d},391,389,9\} & \{98,\text{u},390,392,4\} &
 \{99,\text{d},393,395,6\} & \{99,\text{u},396,394,5\} \\
 \{100,\text{d},399,397,5\} & \{100,\text{u},400,398,5\} &
 \{101,\text{d},403,401,9\} & \{101,\text{u},402,404,5\} &
 \{102,\text{d},407,405,6\} & \{102,\text{u},408,406,5\} \\
 \{103,\text{d},410,412,8\} & \{103,\text{u},409,411,5\} &
 \{104,\text{d},414,416,6\} & \{104,\text{u},413,415,5\} &
 \{105,\text{d},419,417,9\} & \{105,\text{u},420,418,5\} \\
 \{106,\text{d},421,423,5\} & \{106,\text{u},424,422,6\} &
 \{107,\text{d},425,427,7\} & \{107,\text{u},428,426,5\} &
 \{108,\text{d},429,431,5\} & \{108,\text{u},432,430,7\} \\
 \{109,\text{d},435,433,5\} & \{109,\text{u},436,434,6\} &
 \{110,\text{d},439,437,7\} & \{110,\text{u},440,438,5\} &
 \{111,\text{d},441,443,5\} & \{111,\text{u},444,442,6\} \\
 \{112,\text{d},448,446,7\} & \{112,\text{u},447,445,5\} &
 \{113,\text{d},451,449,8\} & \{113,\text{u},452,450,5\} &
 \{114,\text{d},454,456,6\} & \{114,\text{u},453,455,5\} \\
 \{115,\text{d},460,458,9\} & \{115,\text{u},459,457,5\} &
 \{116,\text{d},463,461,8\} & \{116,\text{u},462,464,5\} &
 \{117,\text{d},468,466,9\} & \{117,\text{u},465,467,5\} \\
 \{118,\text{d},472,470,9\} & \{118,\text{u},471,469,5\} &
 \{119,\text{d},473,475,9\} & \{119,\text{u},474,476,5\} &
 \{120,\text{d},479,477,5\} & \{120,\text{u},478,480,6\} \\
 \{121,\text{d},482,484,5\} & \{121,\text{u},481,483,8\} &
 \{122,\text{d},486,488,7\} & \{122,\text{u},485,487,6\} &
 \{123,\text{d},489,491,7\} & \{123,\text{u},490,492,6\} \\
 \{124,\text{d},496,494,8\} & \{124,\text{u},493,495,6\} &
 \{125,\text{d},498,500,8\} & \{125,\text{u},497,499,6\} &
 \{126,\text{d},501,503,6\} & \{126,\text{u},502,504,9\} \\
 \{127,\text{d},506,508,6\} & \{127,\text{u},505,507,8\} &
 \{128,\text{d},510,512,6\} & \{128,\text{u},509,511,9\} &
 \{129,\text{d},516,514,6\} & \{129,\text{u},513,515,8\} \\
 \{130,\text{d},519,517,6\} & \{130,\text{u},518,520,8\} &
 \{131,\text{d},523,521,6\} & \{131,\text{u},522,524,6\} &
 \{132,\text{d},525,527,8\} & \{132,\text{u},528,526,6\} \\
 \{133,\text{d},532,530,8\} & \{133,\text{u},529,531,6\} &
 \{134,\text{d},536,534,8\} & \{134,\text{u},535,533,6\} &
 \{135,\text{d},538,540,7\} & \{135,\text{u},539,537,6\} \\
 \{136,\text{d},542,544,7\} & \{136,\text{u},541,543,6\} &
 \{137,\text{d},548,546,7\} & \{137,\text{u},547,545,7\} &
 \{138,\text{d},551,549,8\} & \{138,\text{u},552,550,9\} \\
 \{139,\text{d},556,554,9\} & \{139,\text{u},553,555,8\} &
 \{140,\text{d},560,558,9\} & \{140,\text{u},559,557,8\} &
 \{141,\text{d},562,564,8\} & \{141,\text{u},561,563,9\} \\
 \{142,\text{d},566,568,9\} & \{142,\text{u},567,565,9\} &
\end{array}
$
\end{center}
}

\begin{figure}[!htb]
\begin{center}
\includegraphics[width=5.8cm]{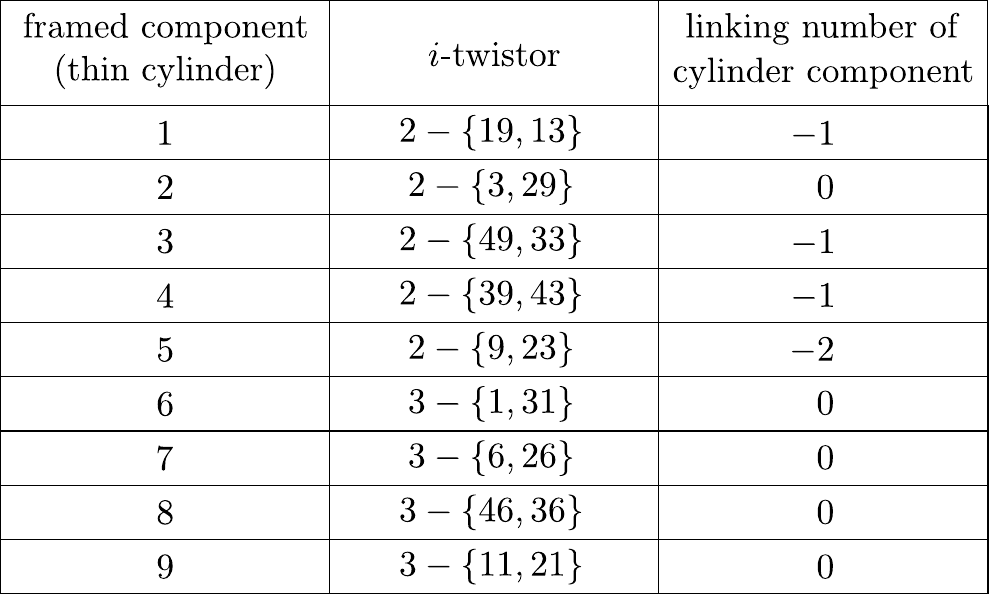}
\end{center}
\end{figure}

\section*{Appendix B: all figures for the $r_5^{24}$-example. This produces
an overview of the data structure and their interrelations illustrating the general case
of the algorithm}

In the following figures, the notation
% $\nabla_{v,1}$ denotes the PL3-face dual of the vertex $v$ in $E\mathcal{H}_1^\star$.
 $\nabla_{v,i+1}$, $i\geq1$, denotes the PL3-face dual of the
vertex $v$ of the input gem obtained after the $i$-th $bp$-move is performed. If after the $i$-th
$bp$-move $\nabla_{v,i}$ does not change, then $\nabla_{v,i+1}=\nabla_{v,i}$. When there is a change,
it is an $\epsilon$-change.

%-----------------------------------
\begin{figure}[!htb]
\begin{center}
\includegraphics[width=15cm]{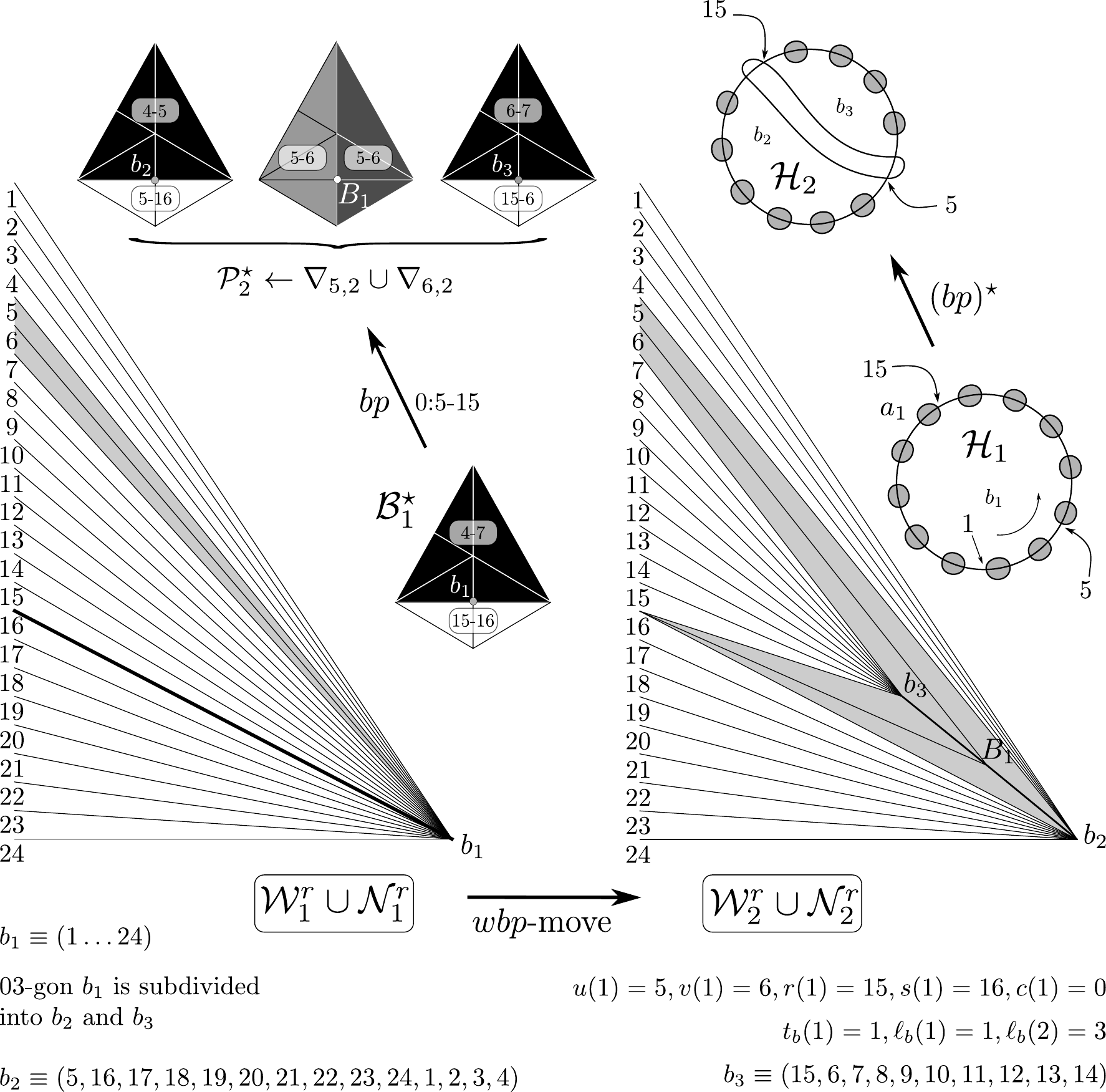} \\
\caption{\sf 
$\mathcal{H}^\star_{2} \leftarrow \mathcal{H}^\star_1 
\cup (\mathcal{P}_{2}^\star \backslash \mathcal{B}_1^\star)$. 
Pillow $\mathcal{P}_{2}^\star \leftarrow 
\nabla_{5,12}\cup \nabla_{6,12}$
($r^{24}_5$-example).}
\label{fig:winglist01}
\end{center}
\end{figure}
%-----------------------------------

%-----------------------------------
\begin{figure}
\begin{center}
\includegraphics[width=15cm]{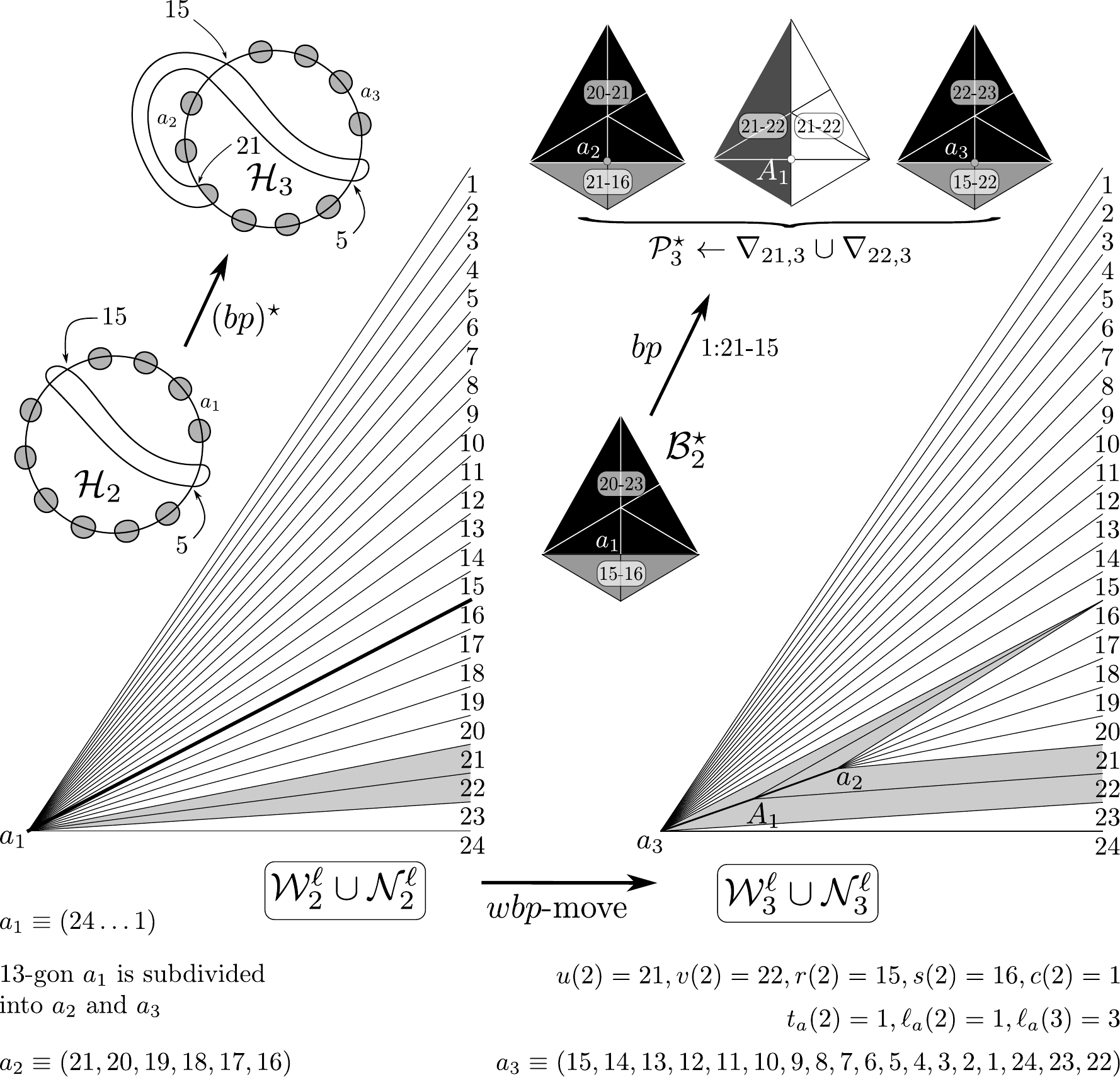} \\
\caption{\sf 
$\mathcal{H}^\star_{3} \leftarrow \mathcal{H}^\star_2 
\cup (\mathcal{P}_{3}^\star \backslash \mathcal{B}_2^\star)$. 
Pillow $\mathcal{P}_{3}^\star \leftarrow 
\nabla_{21,12}\cup \nabla_{22,12}$
($r^{24}_5$-example).}
\label{fig:winglist02}
\end{center}
\end{figure}
%-----------------------------------

%-----------------------------------
\begin{figure}
\begin{center}
\includegraphics[width=15cm]{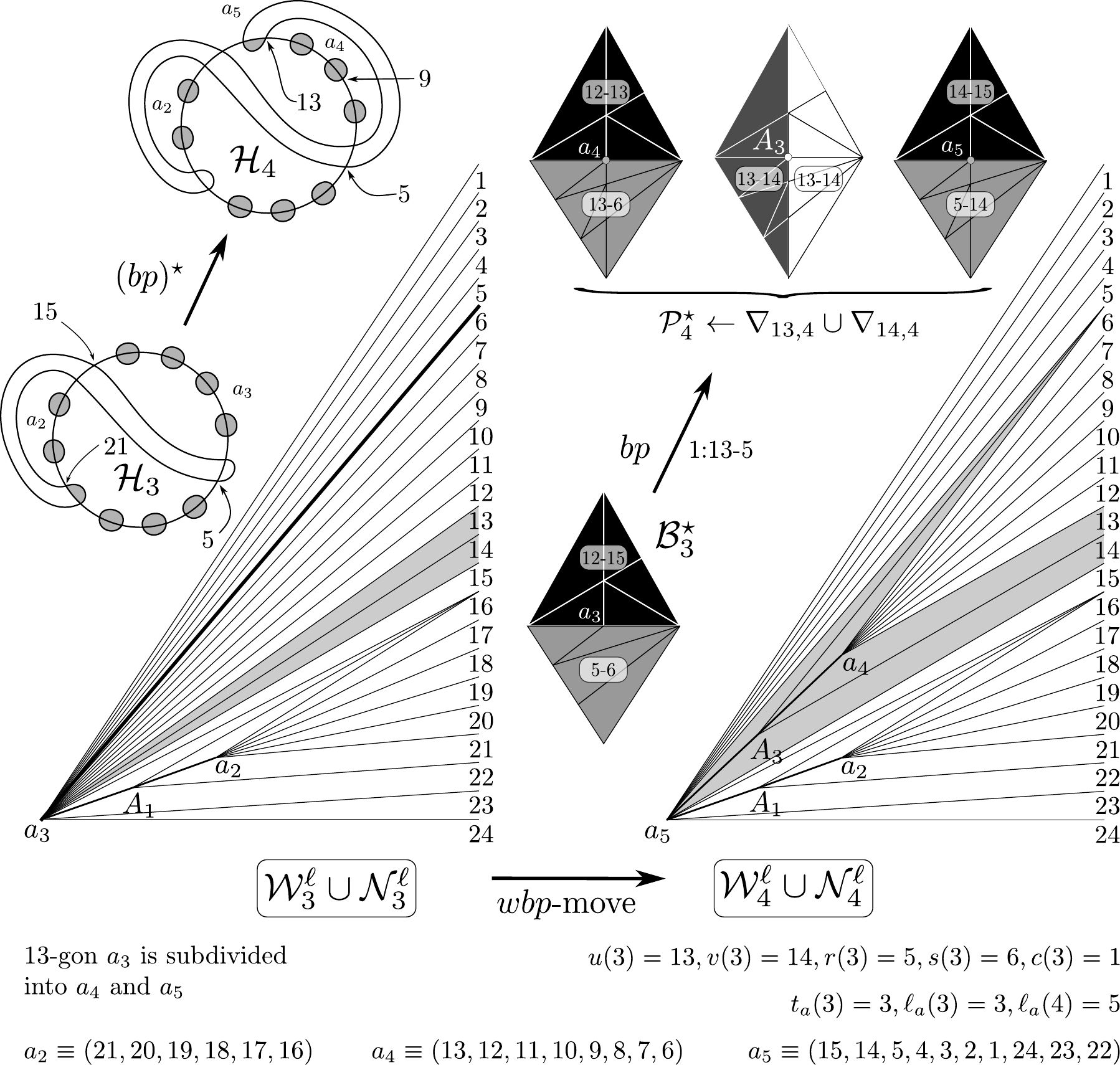} \\
\caption{\sf 
$\mathcal{H}^\star_{4} \leftarrow \mathcal{H}^\star_3 
\cup (\mathcal{P}_{4}^\star \backslash \mathcal{B}_3^\star)$. 
Pillow $\mathcal{P}_{4}^\star \leftarrow 
\nabla_{13,12}\cup \nabla_{14,12}$
($r^{24}_5$-example).}
\label{fig:winglist03}
\end{center}
\end{figure}
%-----------------------------------

%-----------------------------------
\begin{figure}
\begin{center}
\includegraphics[width=15cm]{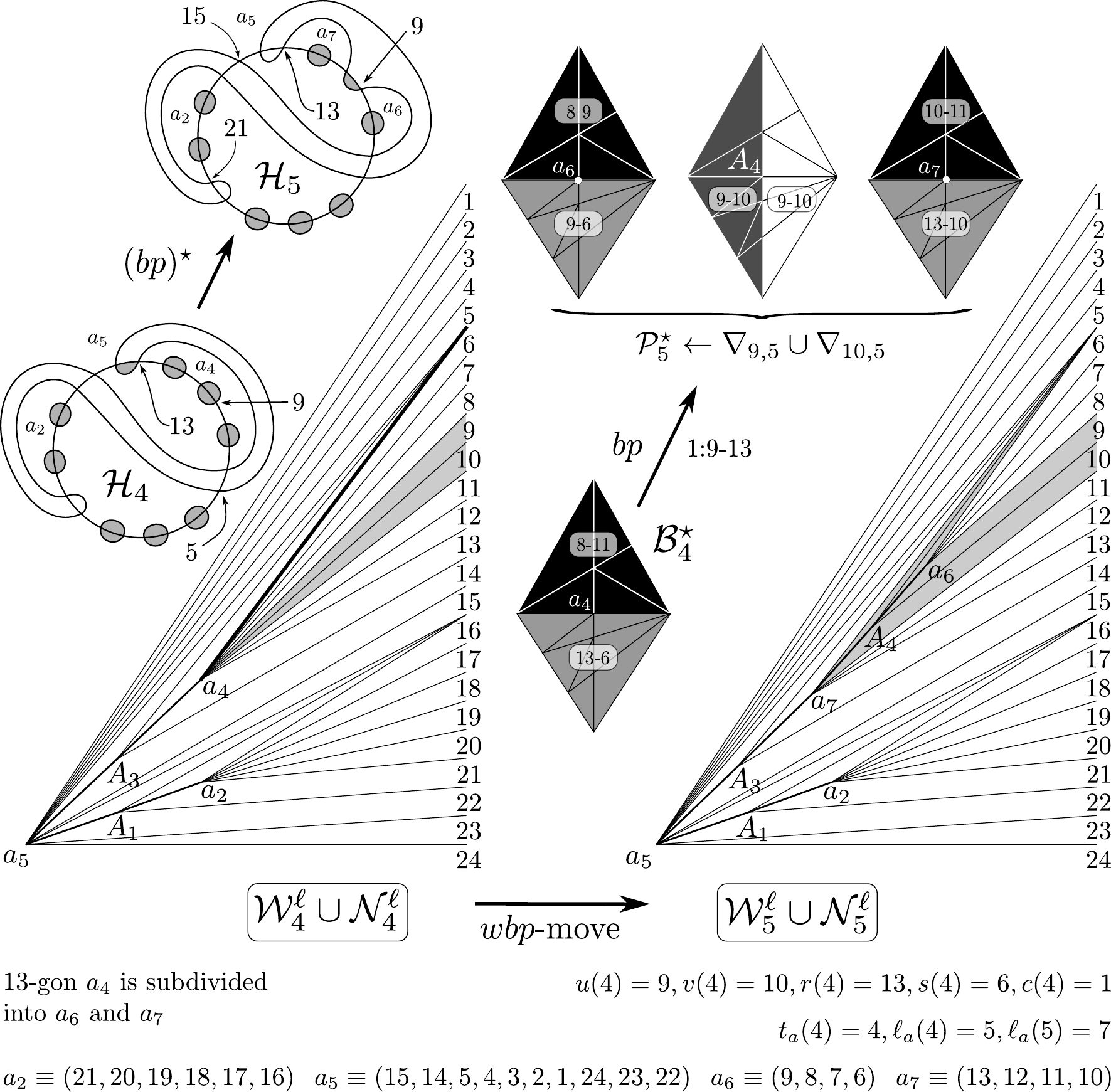} \\
\caption{\sf 
$\mathcal{H}^\star_{5} \leftarrow \mathcal{H}^\star_4 
\cup (\mathcal{P}_{5}^\star \backslash \mathcal{B}_4^\star)$. 
Pillow $\mathcal{P}_{5}^\star \leftarrow 
\nabla_{9,12}\cup \nabla_{10,12}$
($r^{24}_5$-example).}
\label{fig:winglist04}
\end{center}
\end{figure}
%-----------------------------------

%-----------------------------------
\begin{figure}
\begin{center}
\includegraphics[width=15cm]{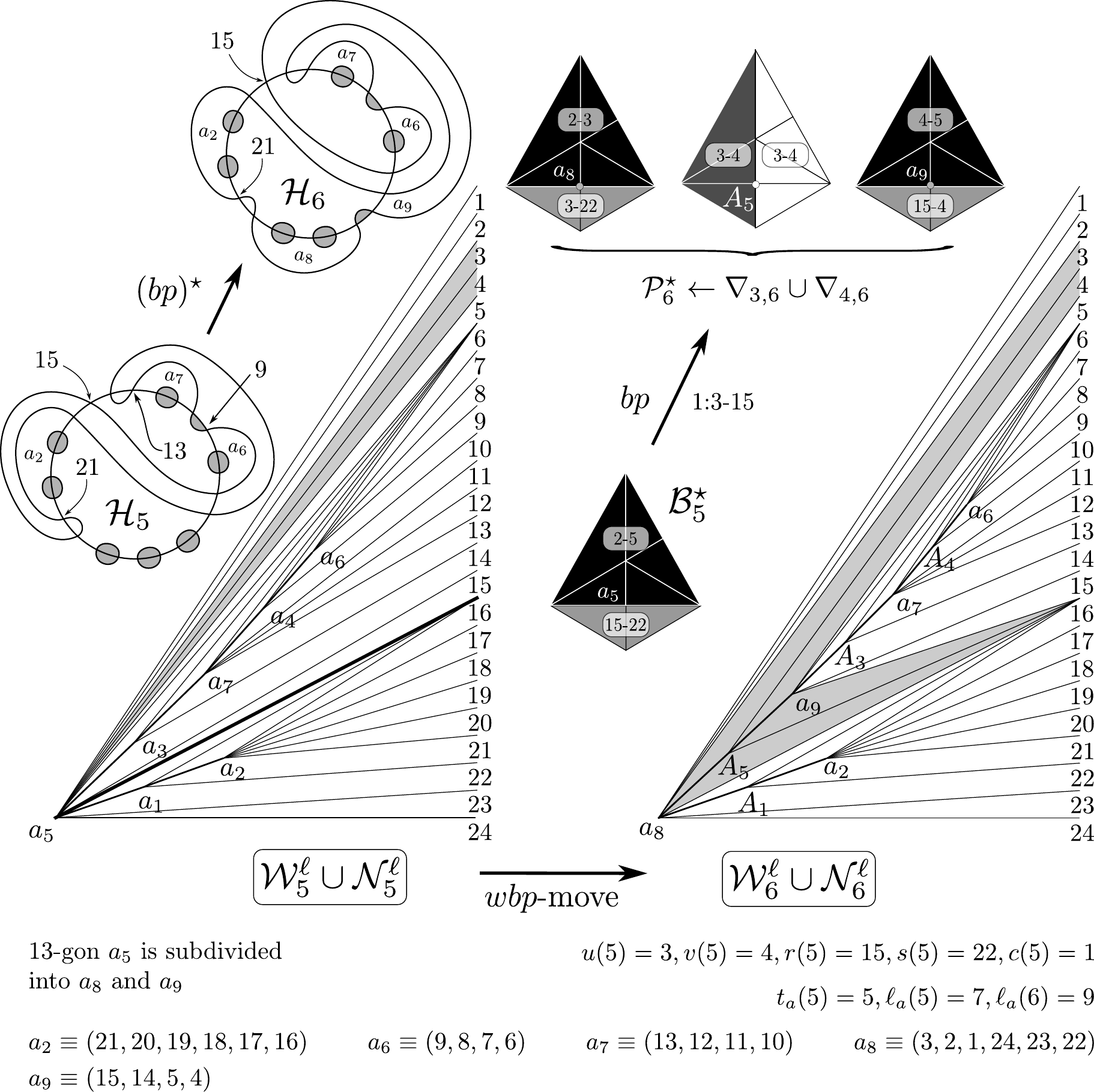} \\
\caption{\sf 
$\mathcal{H}^\star_{6} \leftarrow \mathcal{H}^\star_5 
\cup (\mathcal{P}_{6}^\star \backslash \mathcal{B}_5^\star)$. 
Pillow $\mathcal{P}_{6}^\star \leftarrow 
\nabla_{3,12}\cup \nabla_{4,12}$
($r^{24}_5$-example).}
\label{fig:winglist05}
\end{center}
\end{figure}
%-----------------------------------

%-----------------------------------
\begin{figure}
\begin{center}
\includegraphics[width=15cm]{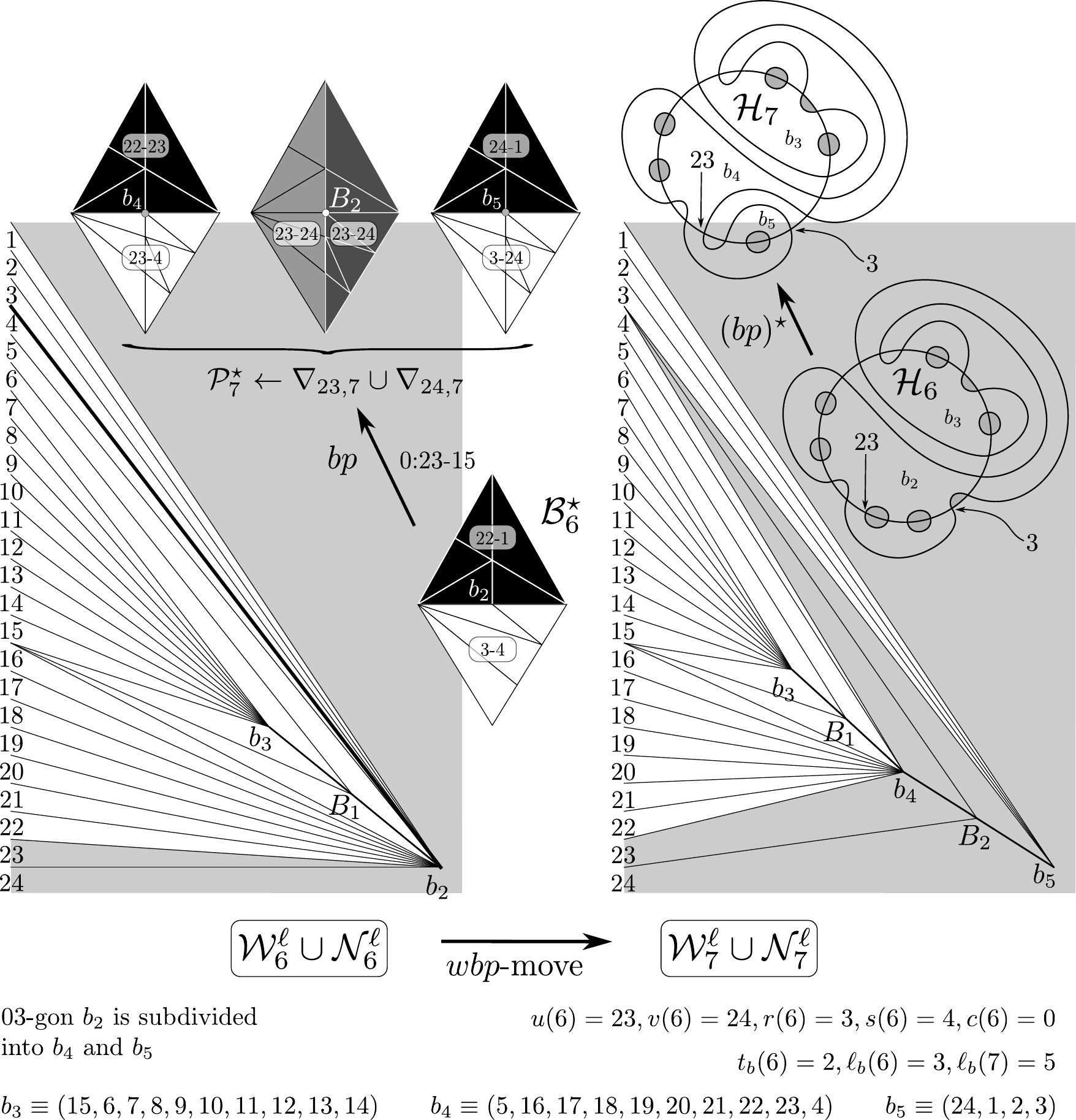} \\
\caption{\sf 
$\mathcal{H}^\star_{7} \leftarrow \mathcal{H}^\star_6 
\cup (\mathcal{P}_{7}^\star \backslash \mathcal{B}_6^\star)$. 
Pillow $\mathcal{P}_{7}^\star \leftarrow 
\nabla_{23,12}\cup \nabla_{24,12}$
($r^{24}_5$-example).}
\label{fig:winglist06}
\end{center}
\end{figure}
%-----------------------------------

%-----------------------------------
\begin{figure}
\begin{center}
\includegraphics[width=15cm]{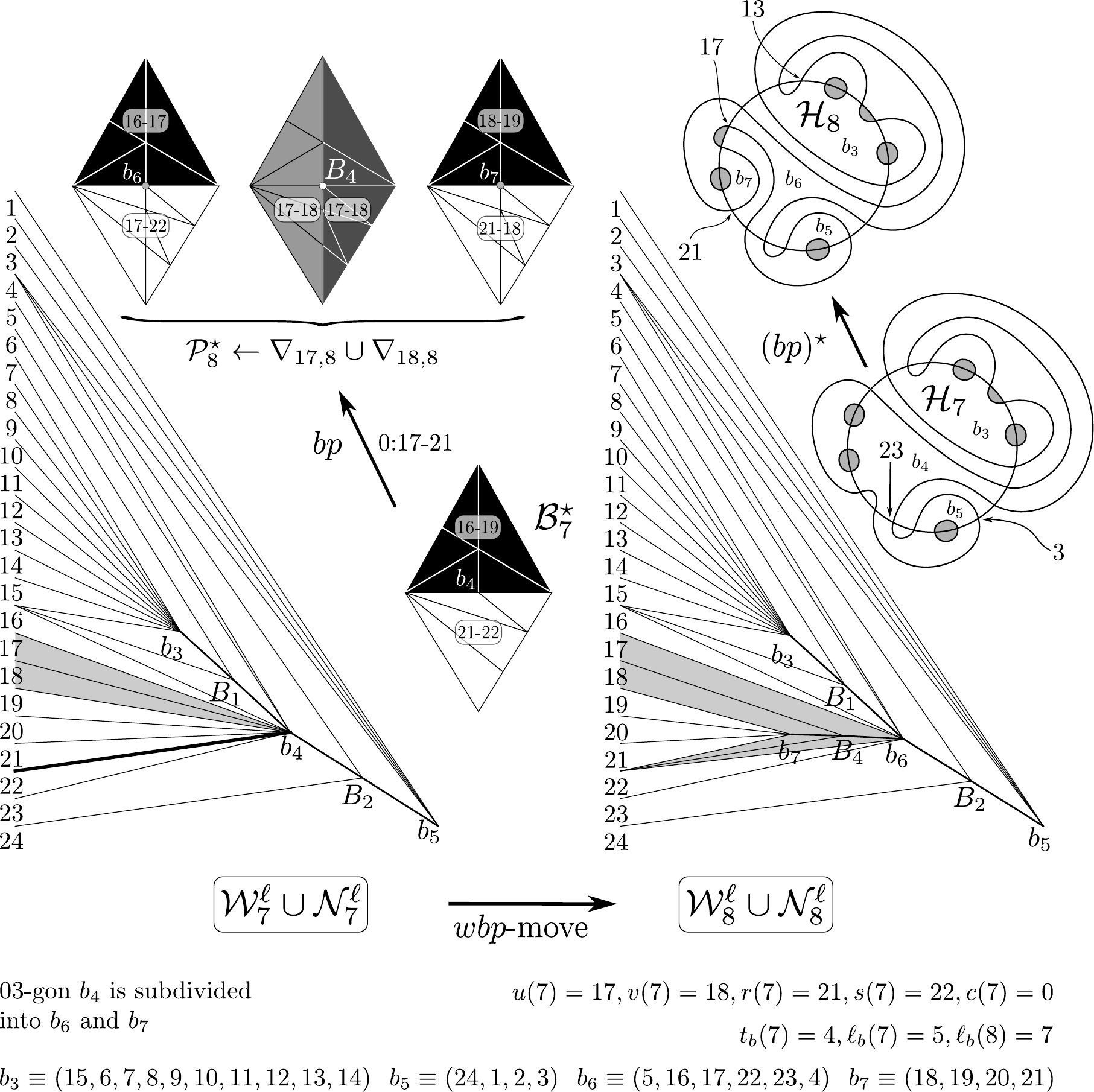} \\
\caption{\sf 
$\mathcal{H}^\star_{8} \leftarrow \mathcal{H}^\star_7 
\cup (\mathcal{P}_{8}^\star \backslash \mathcal{B}_7^\star)$. 
Pillow $\mathcal{P}_{8}^\star \leftarrow 
\nabla_{17,12}\cup \nabla_{18,12}$
($r^{24}_5$-example).}
\label{fig:winglist07}
\end{center}
\end{figure}
%-----------------------------------

%-----------------------------------
\begin{figure}
\begin{center}
\includegraphics[width=15cm]{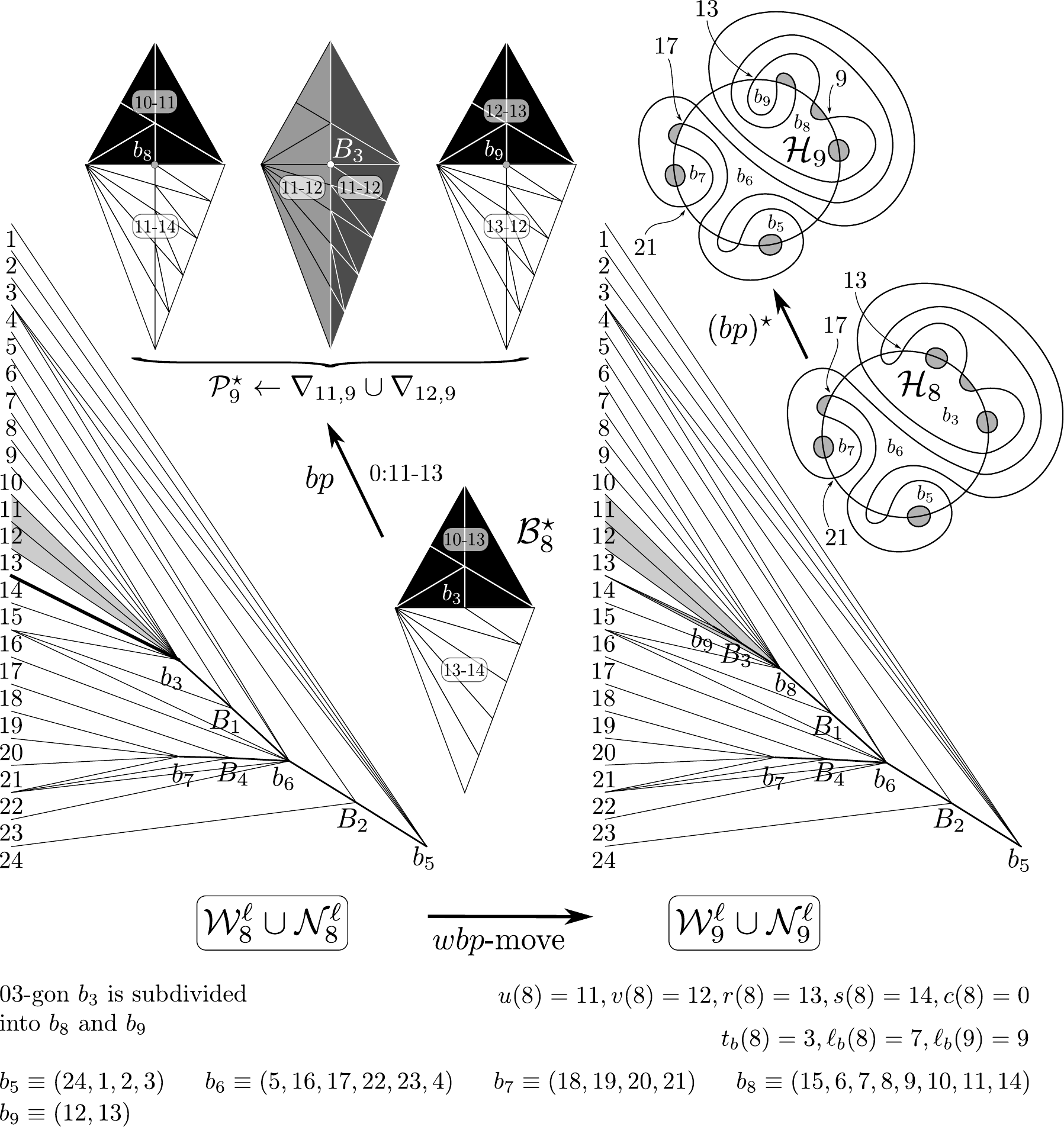} \\
\caption{\sf 
$\mathcal{H}^\star_{9} \leftarrow \mathcal{H}^\star_8 
\cup (\mathcal{P}_{9}^\star \backslash \mathcal{B}_8^\star)$. 
Pillow $\mathcal{P}_{9}^\star \leftarrow 
\nabla_{11,12}\cup \nabla_{12,12}$
($r^{24}_5$-example).}
\label{fig:winglist08}
\end{center}
\end{figure}
%-----------------------------------

%-----------------------------------
\begin{figure}
\begin{center}
\includegraphics[width=15cm]{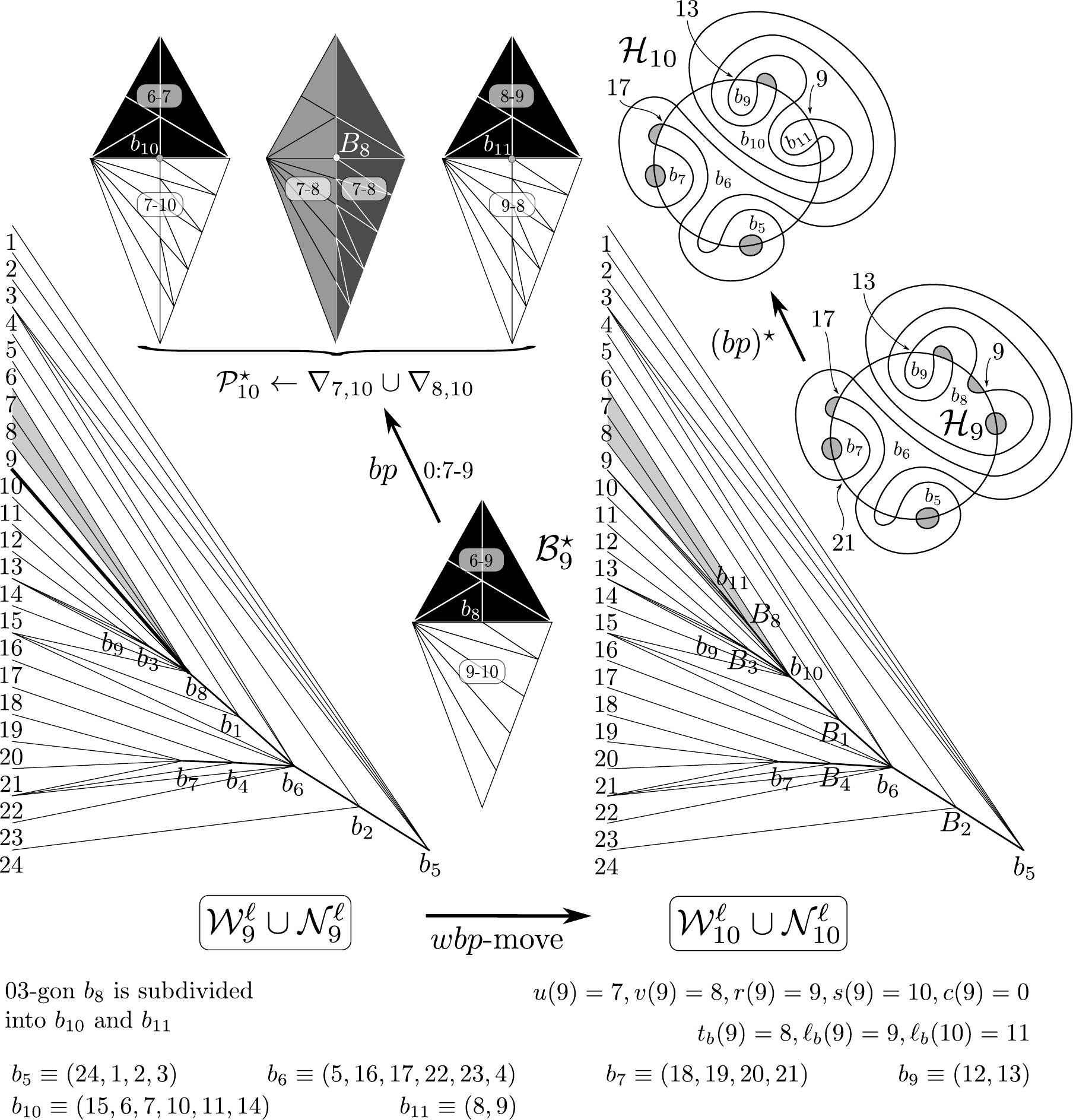} \\
\caption{\sf 
$\mathcal{H}^\star_{10} \leftarrow \mathcal{H}^\star_9 
\cup (\mathcal{P}_{10}^\star \backslash \mathcal{B}_9^\star)$. 
Pillow $\mathcal{P}_{10}^\star \leftarrow 
\nabla_{7,12}\cup \nabla_{8,12}$
($r^{24}_5$-example).}
\label{fig:winglist09}
\end{center}
\end{figure}
%-----------------------------------

%-----------------------------------
\begin{figure}
\begin{center}
\includegraphics[width=15cm]{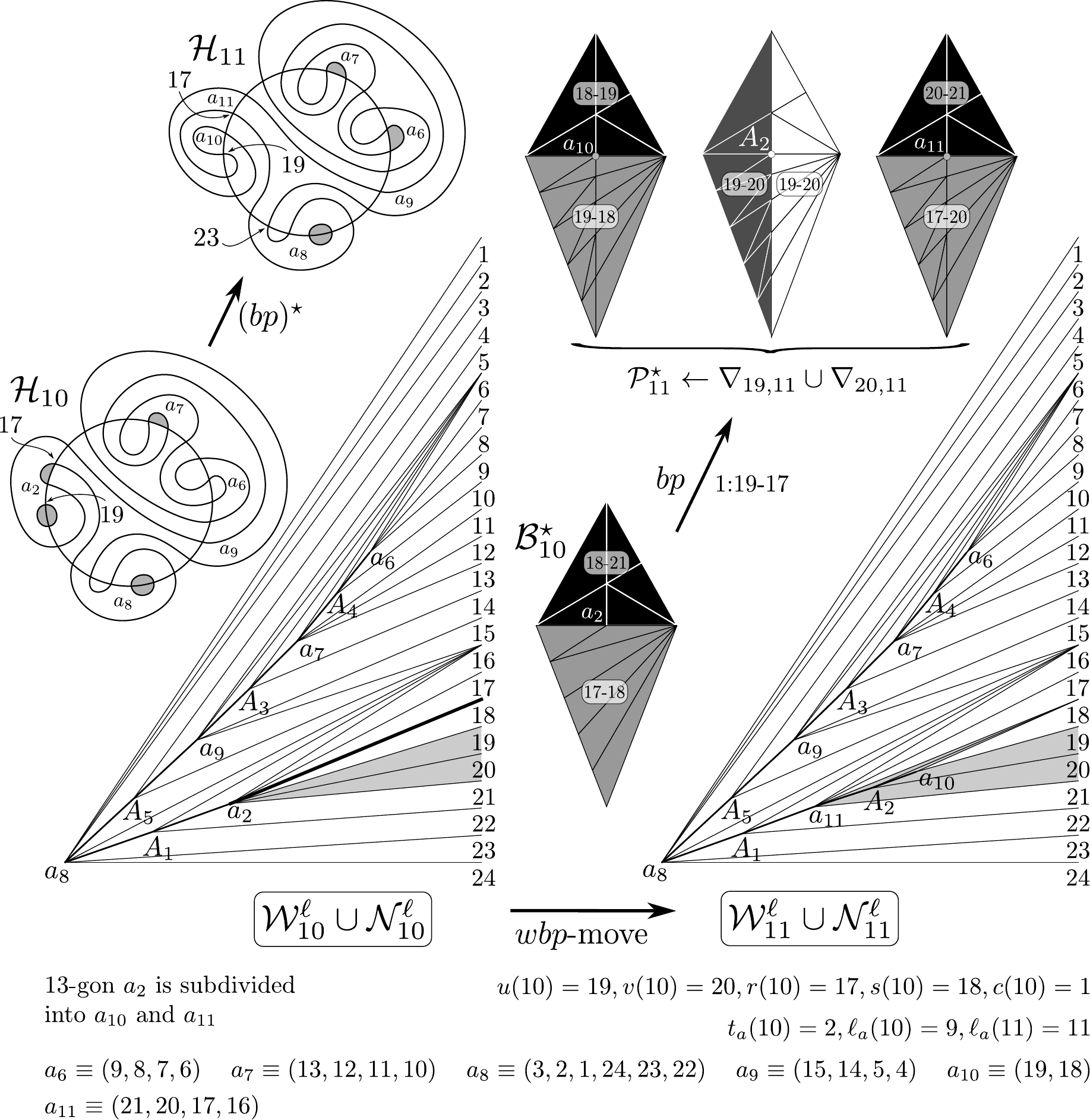} \\
\caption{\sf 
$\mathcal{H}^\star_{11} \leftarrow \mathcal{H}^\star_{10}
\cup (\mathcal{P}_{11}^\star \backslash \mathcal{B}_{10}^\star)$.
Pillow $\mathcal{P}_{11}^\star \leftarrow 
\nabla_{19,12}\cup \nabla_{20,12}$
($r^{24}_5$-example).}
\label{fig:winglist10}
\end{center}
\end{figure}
%-----------------------------------

%-----------------------------------
\begin{figure}
\begin{center}
\includegraphics[width=15cm]{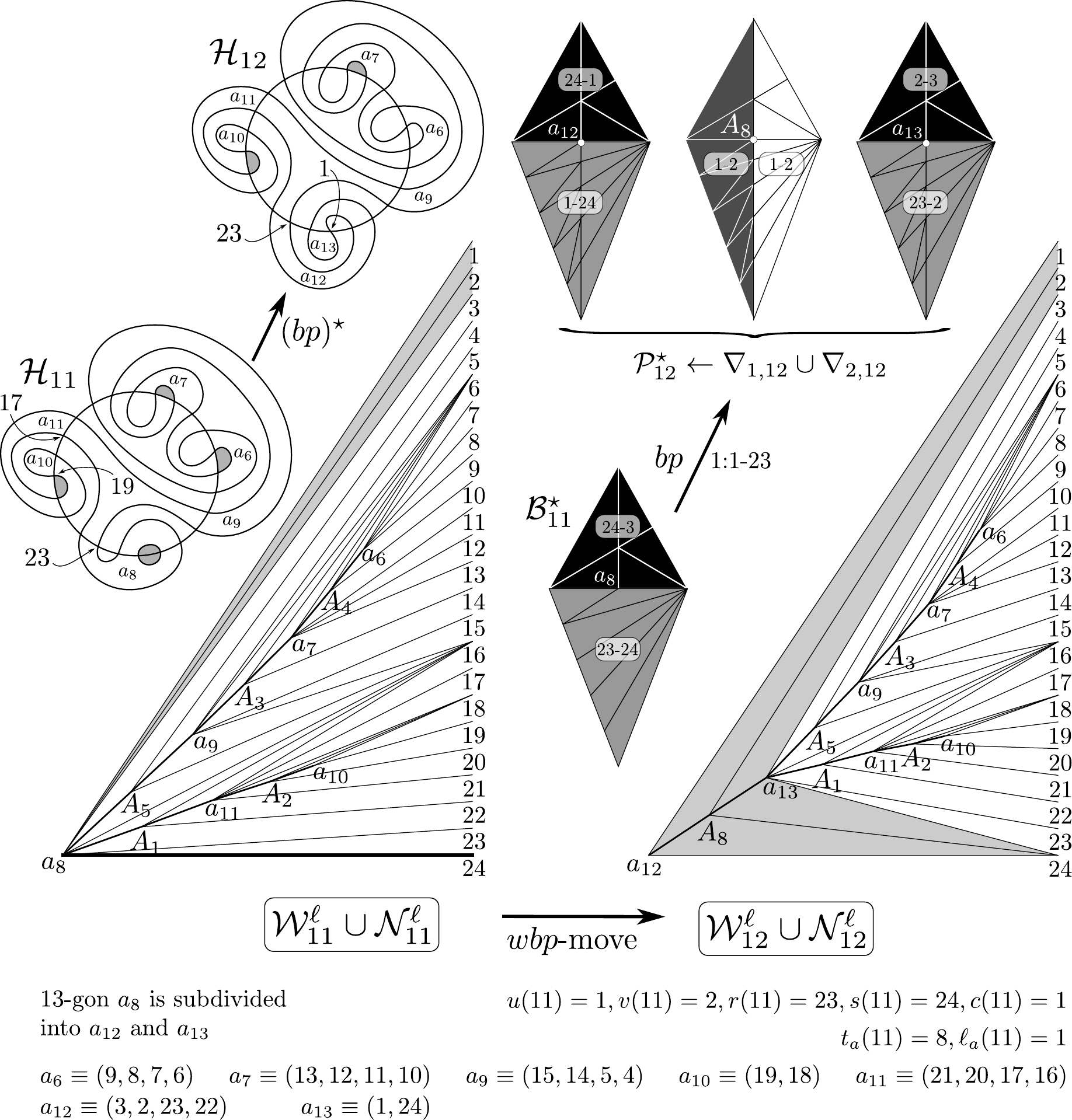} \\
\caption{\sf 
$\mathcal{H}^\star_{12} \leftarrow \mathcal{H}^\star_{11} 
\cup (\mathcal{P}_{12}^\star \backslash \mathcal{B}_{11}^\star)$. 
Pillow $\mathcal{P}_{12}^\star \leftarrow \nabla_{1,12}\cup \nabla_{2,12}$
($r^{24}_5$-example).}
\label{fig:winglist11}
\end{center}
\end{figure}
%-----------------------------------

%-----------------------------------
%\begin{figure}
%\begin{center}
%\includegraphics[width=15cm]{A.figs/bpandwinglist12.pdf} \\
%\caption{\sf 
%$\mathcal{H}^\star_{12} \leftarrow \mathcal{H}^\star_{11} 
%\cup (\mathcal{P}_{12}^\star \backslash \mathcal{B}_{11}^\star)$ 
%and wings (without nervures) of $\mathcal{H}^\star_{12}$
%($r^{24}_5$-example).}
%\label{fig:winglist12}
%\end{center}
%\end{figure}

%-----------------------------------
\bibliographystyle{plain}
%\bibliographystyle{is-alpha}
%\addcontentsline{toc}{bibliografia}{\MakeTextUppercase{Referências Bibliográficas}}
%\bibliography{d:/slsl\3.DadosSostenes.35.ArtigosLivros.bibtexGoogleScholar/bibtexIndex.bib} % bib file is slsl.bib
%\bibliography{~/home/ricardo/Dropbox/35.ArtigosLivros.bibtexGoogleScholar/bibtexIndex.bib}
\bibliography{bibtexIndex.bib}
%\bibliography{slsl}

\vspace{10mm}
\begin{center}

\hspace{7mm}
\begin{tabular}{l}
   S\'ostenes L. Lins\\
   Centro de Inform\'atica, UFPE \\
   Recife--PE \\
   Brazil\\
   sostenes@cin.ufpe.br
\end{tabular}
\hspace{20mm}
\begin{tabular}{l}
   Ricardo N. Machado\\
   Núcleo de Formação de Docentes, UFPE\\
   Caruaru--PE \\
   Brazil\\
   ricardonmachado@gmail.com
\end{tabular}

\end{center}

\end{document}